\documentclass[12pt]{amsart}
\usepackage{amsmath}
\usepackage{amsthm}
\usepackage{amsfonts}
\usepackage{amssymb}
\usepackage[all,cmtip]{xy}           
\usepackage{bm}
\usepackage{bbm}
\usepackage{bbding}
\usepackage{txfonts}
\usepackage{amscd}
\usepackage{xspace}
\usepackage[shortlabels]{enumitem}
\usepackage{ifpdf}
\usepackage{graphicx}
\usepackage[all]{xy}
\ifpdf
  \usepackage[colorlinks,final,backref=page,hyperindex]{hyperref}
\else
  \usepackage[colorlinks,final,backref=page,hyperindex,hypertex]{hyperref}
\fi
\usepackage[active]{srcltx}
\usepackage{tikz-cd}
\usepackage{tikz}

\usetikzlibrary{positioning,fit}
\makeatletter

\newtheorem{thm}{Theorem}[section]
\newtheorem{lem}[thm]{Lemma}
\newtheorem{cor}[thm]{Corollary}
\newtheorem{pro}[thm]{Proposition}
\theoremstyle{definition}
\newtheorem{ex}[thm]{Example}
\newtheorem{rmk}[thm]{Remark}
\newtheorem{defi}[thm]{Definition}

\newcommand{\nc}{\newcommand}
\newcommand{\delete}[1]{}

\newcommand{\CYB}{\operatorname{CYB}}
\newcommand{\Alt}{\operatorname{Alt}}

\nc{\mlabel}[1]{\label{#1}}  
\nc{\mcite}[1]{\cite{#1}}  
\nc{\mref}[1]{\ref{#1}}  
\nc{\meqref}[1]{\eqref{#1}}  
\nc{\mbibitem}[1]{\bibitem{#1}} 

\delete{
\nc{\mlabel}[1]{\label{#1}{\hfill \hspace{1cm}{\bf{{\ }\hfill(#1)}}}}
\nc{\mcite}[1]{\cite{#1}{{\bf{{\ }(#1)}}}}  
\nc{\mref}[1]{\ref{#1}{{\bf{{\ }(#1)}}}}  
\nc{\meqref}[1]{\eqref{#1}{{\bf{{\ }(#1)}}}}  
\nc{\mbibitem}[1]{\bibitem[\bf #1]{#1}} 
}

\DeclareMathOperator{\im}{Im}

\newcommand {\emptycomment}[1]{}

\delete{

}

\nc{\oprn}{\theta}

\delete{
\setlength{\baselineskip}{1.8\baselineskip}

\newcommand{\emptycomment}[1]{}

}

\nc{\calo}{\mathcal{O}}
\nc{\oop}{$\mathcal{O}$-operator\xspace}
\nc{\oops}{$\mathcal{O}$-operators\xspace}
\nc{\mrho}{{\bm{\varrho}}}
\nc{\bfk}{\mathbf{K}}
\nc{\invlim}{\displaystyle{\lim_{\longleftarrow}}\,}
\nc{\ot}{\otimes}

\nc{\eval}[1]{\Big|_{#1}}

\newcommand{\be }{\begin{equation}}
\newcommand{\ee }{\end{equation}}

\nc{\RR}{\mathbb{R}}

\nc{\CC}{\mathbb{C}}

\newcommand{\Ad}{\operatorname{Ad}}

\newcommand{\Gr}{\operatorname{Gr}}

\newcommand{\frkg}{\mathfrak g}
\newcommand{\frkh}{\mathfrak h}

\newcommand{\otp}{\hat{\otimes}_{\pi}}

\newcommand{\br}[1]{   [ \cdot,    \cdot  ]   }
\newcommand{\id}{\mathsf{id}}

\newcommand{\Ker}{\mathrm{Ker}}

\newcommand{\ad}{\mathrm{ad}}

\nc{\CV}{\mathbf{C}}

\NewDocumentEnvironment{Thm}{O{thm} D(){} m}
  {\addtocounter{#1}{-1}%
   \expandafter\renewcommand\csname the#1\endcsname{\ref{#3}}%
   \begin{#1}[#2]}
  {\end{#1}}

\begin{document}

\title[Projective Banach Lie bialgebras, the PYB equation and projective Poisson-Lie groups]{Projective Banach Lie bialgebras, the projective Yang-Baxter equation and projective Poisson-Lie groups }

\author{Zhonghua Li}
\address{School of Mathematical Sciences, Key Laboratory of Intelligent Computing and Applications (Ministry of Education), Tongji University, Shanghai 200092, China}
\email{zhonghua$\_$li@tongji.edu.cn}

\author{Shukun Wang}
\address{School of Mathematics and Big Data, Anhui University of Science and Technology, Huainan 232001, China}
\email{2024093@aust.edu.cn}


\begin{abstract}
In this paper, we first introduce the notion of projective Banach Lie bialgebras as the projective tensor product analogue of Banach Lie bialgebras. Then we consider the completion of the classical Yang-Baxter equation and classical r-matrices,  and propose the notions of the projective Yang-Baxter equation and projective r-matrices. As in the finite-dimensional case, we prove that every quasi-triangular projective r-matrix gives rise to a projective Banach Lie bialgebra. Next adapting Poisson Banach-Lie groups to the projective tensor product setting, we propose the notion of projective Banach Poisson-Lie groups and show that the differentiation of a projective Banach Poisson-Lie group has the projective Banach Lie bialgebra structure. Finally considering bounded $\mathcal{O}$-operators on Banach Lie algebras, we give an equivalent description of triangular projective r-matrices.
\end{abstract}


\keywords{projective Banach Lie bialgebras, the projective Yang-Baxter equation, projective r-matrices, projective Banach Poisson-Lie groups, $\mathcal{O}$-operators. \\
\quad  2020 \emph{Mathematics Subject Classification.} 17B38,  17B62,   22E60, 53D17.\\
Corresponding author: Shukun Wang}

\maketitle

\tableofcontents

\section{Introduction}
\subsection{The motivation}
\subsubsection{Finite-dimensional Lie bialgebras and Poisson Lie groups}

Poisson Lie groups and Lie bialgebras were introduced by V. Drinfeld in \cite{VGD} in the middle of the 1980s. Lie bialgebras are the infinitesimal counterparts of Poisson-Lie groups and are closely related to quantum groups. More precisely, in \cite{D1}, Lie bialgebras were described as the classical limit of quantum groups. Furthermore, the investigation of Lie bialgebras is important in the advancement of integrable systems \cite{CR,F,GO,SK,S2}.

A quasi-triangular Hopf algebra is a Hopf algebra with an invertible element $R$ satisfying some conditions \cite{D1}. The element $R$ is called a quantum R-matrix, which provides a solution of the quantum Yang-Baxter equation. The semi-classical limit of a quantum R-matrix is called a classical r-matrix. It provides an important class of Lie bialgebras, called quasi-triangular Lie bialgebras.

The notion of 
$\mathcal{O}$-operators were introduced and its relationship with classical r-matrices was investigated in \cite{Ku}. More precisely, it was proven that every skew-symmetric classical r-matrix is induced by an $\mathcal{O}$-operator. For more information on the relationship between $\mathcal{O}$-operators and classical r-matrices, we refer the readers to \cite{Bai, B2024}.

\subsubsection{Infinite-dimensional Lie bialgebras}

Unlike the finite-dimensional cases, there is no general relationship between the structures of infinite-dimensional Lie bialgebras and Poisson–Lie groups. However, infinite-dimensional Lie bialgebras are also quantizable \cite{EK,EK1} and are hence closely related to quantum groups. Moreover, there are numerous examples on such structures, for example, the infinite-dimensional Lie bialgebras on Witt algebras and Virasoro algebras \cite{N}.

\subsubsection{Banach Lie bialgebras and Banach Poisson-Lie groups}

An important generalization of infinite-dimensional Lie bialgebras is Banach Lie bialgebras. In \cite{TP}, by adapting the notion of 
Poisson Lie groups and Lie bialgebras to the Banach context, Banach Poisson-Lie groups and Banach Lie bialgebras were introduced. It was demonstrated that the differentiation of a Banach Lie group gives rise to the Banach Lie algebra structure. Furthermore, it was established that every restricted Grassmannian inherits a generalized
Banach Poisson structure from the unitary Banach Lie group, called the Bruhat-Poisson structure.

\subsubsection{Topological Lie bialgebras}

Another important generalization of infinite-dimensional Lie bialgebras is topological Lie bialgebras. In \cite{MX2024}, by considering Lie bialgebras and Main triples in the topological setting, the notion of topological Lie bialgebras was introduced. A topological Lie bialgebra is composed of a topological Lie algebra and a topological Lie coalgebra that satisfies the cocycle condition. Given a Lie algebra $\mathfrak{g}$, the completion of the classical Yang-Baxter equation on $\mathfrak{g}[\![x]\!]$ with respect to the projective topology is called the formal classical Yang-Baxter equation. A solution of the formal classical Yang-Baxter equation on $\mathfrak{g}[\![x]\!]$ is called a formal r-matrix, which gives rise to a topological Lie bialgebra structure on $\mathfrak{g}[\![x]\!]$.

\subsubsection{The goal of this paper}

Building on the aforementioned works \cite{TP,MX2024}, this paper focuses on the study of Banach Lie bialgebras and Banach Poisson-Lie groups in the projective tensor product setting and their connections, aims to establish the relationship between such Banach Lie bialgebras and the completion of the classical Yang-Baxter equation with respect to the projective tensor norm.

\subsection{Main results}
\subsubsection{Projective Banach Lie bialgebras and the projective Yang-Baxter equation}

Considering Banach Lie bialgebras in the projective tensor product setting, we introduce the notion of projective Banach Lie bialgebras in Section \ref{2}. The completion of the classical Yang-Baxter equation with respect to the projective tensor norm given in Section \ref {3} leads to the notion of the projective Yang-Baxter equation. A solution of the projective Yang-Baxter equation is called a projective r-matrix. We establish the relationship between projective r-matrices and projective Banach Lie bialgebras in Section \ref{3}. Below is a summary of some of our main results.

\begin{thm}[Theorem \ref{T1} and Corollary \ref{QP}]
Let $\mathfrak{g}$ be a Banach Lie algebra. 
\begin{itemize}
\item[{\rm(a)}]  If there is a sequence of symmetric elements $(s_n\in \mathfrak{g}\otimes\mathfrak{g})_{n=0}^{\infty}$ such that $\lim_{n\to +\infty} \ad^{(2)}_{x} s_n=0$ in the projective norm for any $x\in \mathfrak{g}$ and  $(s_n\in \mathfrak{g}\otimes\mathfrak{g})_{n=0}^{\infty}$ is bounded in the projective norm, then $\lim_{n\to +\infty} \ad^{(3)}_{x}\langle s_n,s_n\rangle=0$ in the projective norm.
\item[{\rm(b)}] Let $r=a+s\in \mathfrak{g}\hat{\otimes}_{\pi} \mathfrak{g},$ where $a$ is skew-symmetric and s is symmetric and $\ad$-invariant. Let $(r_n\in \mathfrak{g}\otimes \mathfrak{g})_{n=0}^{\infty}$ be a sequence such that
$\lim_{n\to +\infty} r_n=r$ in the projective norm and $a_n,s_n$ be the skew-symmetric and the symmetric parts of $r_n$ for any $n\geq 0$ respectively. If $\lim_{n\to +\infty}\langle r_n,r_n\rangle=0,$ then $\lim_{n\to+\infty}\ad^{(3)}_{x}[\![a_n,a_n]\!]=0$ in the projective norm for any $x\in \mathfrak{g}$.
\item[{\rm(c)}] Let $r\in \frkg\otp \frkg$ be a projective {\rm r}-matrix, $a$ and $s$ be the skew-symmetric and symmetric part respectively. If $s$ is $\ad$-invariant, then $(\frkg,dr)$ is a projective Banach Lie bialgebra.
\end{itemize}
\end{thm}

\subsubsection{Projective Banach Lie groups}

As the projective tensor version of Banach Poisson groups, we introduce projective Poisson-Lie groups in Section \ref{4}. It is natural to consider the differentiation of such structures. The following is our main result in Section \ref{4}.

\begin{thm}[Theorem \ref{T2}]
 Let $(G,P)$ be a projective Banach Poisson Lie group and $\Delta:G\to \barwedge^{2} TM$ be the operator given by 
 $$ \Delta(a)=R_{a^{-1}}P_{a},\quad \forall a\in G.$$
 Then $(\mathfrak{g},T_{e} \Delta)$ is a projective Banach Lie algebra, where $\mathfrak{g}$ is the Banach Lie algebra of $G$, $T_{e}\Delta$ is the differential of $\Delta$ at the unit element $e$ and $\barwedge^{2} TM$ is the vector bundle over $M$ whose fiber over $a\in M$ lies in the Banach space of skew-symmetric elements of $T_{a} M\otp T_{a} M$.
\end{thm}

\subsubsection{Bounded $\mathcal{O}$-operators on Banach Lie algebras and projective Banach Lie bialgebras}

Using bounded $\mathcal{O}$-operators on Banach Lie algebras, we give an equivalent characterization of the triangular projective r-matrices in Section \ref{5}.

\subsection{Structure of this paper}

The paper is organized as follows. In Section \ref{2}, we first recall the projective tensor products and nuclear bilinear forms on Banach spaces. We also recall the duality pairings between Banach spaces and Banach Poisson spaces in Section \ref{22}. Then considering Banach Poisson Lie bialgebras in the projective tensor product setting, we introduce the notion of projective Banach Lie bialgebras and concrete Banach Lie bialgebras, and establish their connections with Banach Lie bialgebras in Section \ref{23}. Moreover, we investigate the relationship between Banach Manin triples and projective Banach Lie bialgebras in Section \ref{24}.

In Section \ref{3}, as the completion of the classical Yang-Baxter equation and classical r-matrices, we first introduce the notion of the projective Yang-Baxter equation and projective r-matrices respectively. We give some basic examples and prove that every quasi-triangular r-matrix gives rise to the projective Banach Lie bialgebra structures in Section \ref{3.2}. We adapt the notion of coboundary Lie bialgebras to the Banach setting and propose the notion of coboundary Banach Lie bialgebras in Section \ref{33}.

In Section \ref{4}, as the projective analogue of Banach Poisson manifolds, we give the notion of projective Banach Poisson manifolds and investigate their connections with Banach Poisson manifolds. Then in Section \ref{42}, we introduce the notion of projective Banach Poisson-Lie groups as the projective tensor product analogue of Banach Poisson Lie groups. We prove that the Lie algebra of a projective Banach Lie group carries the projective Banach Lie bialgebra structure.

Finally in Section \ref{5}, considering $\mathcal{O}$-operators in the Banach setting, we introduce the notion of bounded $\mathcal{O}$-operators on Banach spaces. We give an equivalent description of triangular projective r-matrices using bounded $\mathcal{O}$-operators. We also consider bounded $\mathcal{O}$-operators of the projective tensor form and give an equivalent characterization of such operators.

In this paper, we take the base field to be $\mathbb{K},$ where $\mathbb{K}\in \{\mathbb{R},\mathbb{C}\}.$

\section{Projective Banach Lie bialgebras}\label{2}

In this section, we recall some concepts that we will need for our construction and introduce the notion of projective Banach Lie bialgebras. First we recall projective tensors and nuclear bilinear forms on Banach spaces in Section \ref{21}. We then introduce the duality pairings and Banach Poisson spaces in Section \ref{22}. In Section \ref{23}, we consider the notion of Banach Lie bialgebras and its analogy in the context of projective tensor products, leading to the concept of projective Banach Lie bialgebras. Finally in Section \ref{24}, we establish the relationship between Banach Manin triples and projective Banach Lie bialgebras.

\subsection{Projective tensor products and nuclear bilinear forms on Banach spaces}\label{21}

\subsubsection{Continuous multilinear maps and bounded operators}

First let us recall the definition of continuous maps on Banach spaces.

Let $\mathfrak{g}_1,\ldots,\mathfrak{g}_{n}$ and $\mathfrak{h}$ be Banach spaces. Denote by $\mathcal{L}(\mathfrak{g}_1,\ldots,\mathfrak{g}_{n};\mathfrak{h})$ the space of continuous multilinear maps from $\mathfrak{g}_1\times\cdots\times\mathfrak{g}_{n}$ to $\mathfrak{h}$. The norm of an element $S\in\mathcal{L}(\mathfrak{g}_1,\ldots,\mathfrak{g}_{n};\mathfrak{h})$ is given by 
$$\|S\|=\sup\left\{\|S(x_1,x_2,\ldots,x_n)\|,\ x_i \in B_{\mathfrak{g}_i   }\right\},$$ 
where $B_{\mathfrak{g}_i}$ is the unit ball of $\mathfrak{g}_i.$

Denote by $\mathcal{B}(\mathfrak{g}_1,\mathfrak{g}_2)$ the Banach space of bounded linear operators from $\mathfrak{g}_1$ to $\mathfrak{g}_2$.

\subsubsection{The projective tensor product on Banach spaces }

Then let us recall the definition of projective tensor product on Banach spaces.

Let $\mathfrak{g}_1$ and $\mathfrak{g}_2$ be Banach spaces. The projective tensor norm $||\cdot||_{\pi}$ on $\mathfrak{g}_1\otimes \mathfrak{g}_2$ is defined as
$$||u||_{\pi}= \operatorname{inf}\left\{\sum_{i=1}^{m} ||x_i||\;||y_i||:u=\sum_{i=1}^{m} x_{i}\otimes y_{i}\right\}.$$
The tensor product endowed with the projective tensor norm is denoted by $\mathfrak{g}_1\otimes_{\pi} \mathfrak{g}_2$. The completion of $\mathfrak{g}_1\otimes_{\pi} \mathfrak{g}_2$ is denoted by $\mathfrak{g}_1\hat{\otimes}_{\pi} \mathfrak{g}_2,$ which is called the projective tensor product of $\mathfrak{g}_1$ and $\mathfrak{g}_2.$ 

The definition of projective tensor norm can be generalized to the case for $n>2$. For example, let $\mathfrak{g}_1,\mathfrak{g}_2$ and $\mathfrak{g}_3$ be Banach spaces. For $u\in\mathfrak{g}_1\otimes\mathfrak{g}_2\otimes\mathfrak{g}_3$, define  
$$||u||_{\pi}^{\otimes 3}= \operatorname{inf}\left\{\sum_{i=1}^{m} ||x_i||\;||y_i||\;||z_i||:u=\sum_{i=1}^{m} x_{i}\otimes y_{i}\otimes z_{i}\right\}.$$ 
Then we get the projective tensor product of $\mathfrak{g}_1,\mathfrak{g}_2$ and $\mathfrak{g}_3$, which is denoted by $\mathfrak{g}_1\hat{\otimes}_{\pi} \mathfrak{g}_2\hat{\otimes}_{\pi} \mathfrak{g}_3$ similarly. It follows from \cite{GR} that as Banach spaces,
 $\mathfrak{g}_1\hat{\otimes}_{\pi} \mathfrak{g}_2\hat{\otimes}_{\pi} \mathfrak{g}_3,$ 
 $(\mathfrak{g}_1\hat{\otimes}_{\pi} \mathfrak{g}_2)\hat{\otimes}_{\pi} \mathfrak{g}_3$ and $\mathfrak{g}_1\hat{\otimes}_{\pi} (\mathfrak{g}_2\hat{\otimes}_{\pi} \mathfrak{g}_3)$ are isometric isomorphic.
 
\begin{rmk}
By \cite[Proposition 2.8]{RYAN}, every element $a\in \frkg_1\otp \frkg_2$ can be written as $\sum_{i=1}^{\infty}u_{i}\otimes v_{i}.$ This property plays a crucial role in subsequent proofs, while other tensor norms on Banach spaces, such as the injective tensor product norm, may not possess such property. Hence in this paper, we mainly focus on the projective tensor norm.
\end{rmk}

The following proposition is a well-known result on projective tensor products.

\begin{pro}[\cite{RYAN}]\label{UPTP}
Let $\mathfrak{g}_1,\mathfrak{g}_2$ and $\mathfrak{g}_3$ be Banach spaces. Then for any continuous bilinear map $B$ from $\mathfrak{g}_1\times \mathfrak{g}_2$ to $\mathfrak{g}_3$, there exists a unique bounded linear operator $\hat{B}:\mathfrak{g}_1\hat{\otimes}_{\pi} \mathfrak{g}_2\to \mathfrak{g}_3$ such that for any $x\in \mathfrak{g}_1$ and $y\in \mathfrak{g}_2$, 
$\hat{B}(x\otimes y)=B(x,y).$ Furthermore, the correspondence $B \leftrightarrow \hat{B}$ is an isometric isomorphism between Banach spaces $\mathcal{L}(\mathfrak{g}_1, \mathfrak{g}_2;\mathfrak{g}_3)$ and $\mathcal{B}(\mathfrak{g}_1\hat{\otimes}_{\pi}\mathfrak{g}_2,\mathfrak{g}_3)$.
\end{pro}

\subsubsection{Nuclear bilinear forms}

Let $\mathfrak{g}_1$ and $\mathfrak{g}_2$ be Banach spaces. By \cite{RYAN}, we know that there is a natural algebraic embedding from $\mathfrak{g}^{\ast}_1\otimes \mathfrak{g}^{\ast}_2$ to $\mathcal{L}(\mathfrak{g}_1, \mathfrak{g}_2;\mathbb{K})$. Then the extension of this embedding to the completions gives an operator $$J:\mathfrak{g}^{\ast}_1\hat{\otimes}_{\pi} \mathfrak{g}^{\ast}_2\to \mathcal{L}(\mathfrak{g}_1,\mathfrak{g}_2;\mathbb{K}).$$

A map in $\mathcal{L}(\mathfrak{g}_1,\mathfrak{g}_2;\mathbb{K})$ is called nuclear if it lies in the image of $J$. Equivalently, a bilinear form $S$ is nuclear if and only if there exists bounded sequences $\left(\alpha_n\right)$ and $\left(\beta_n\right)$ in $\mathfrak{g}_1^{\ast}$ and $\mathfrak{g}_2^{\ast}$ respectively such that $\sum_{n=1}^{\infty} \| \alpha_n\|\  \|\beta_n \|<\infty$, and for any $x\in\mathfrak{g}_1, y\in\mathfrak{g}_2$, 
$$S(x,y)=\sum_{n=1}^{\infty} \alpha_{n}(x) \beta_{n}(y).$$ 
An expression of this form $\sum_{n=1}^{\infty} \alpha_n\otimes \beta_n$ is called a nuclear expression of $S$. The nuclear norm of $S$ is defined to be 
$$\|S\|_{N}=\operatorname{inf}\left\{\sum_{n=1}^{\infty} 
\|\alpha_n\|\ \|\beta_n\|:\sum_{n=1}^{\infty} \alpha_n\otimes \beta_n\text{\;is a nuclear expression of\;}S\right\}.$$ 
It is easy to check that 
\begin{equation}\label{CP}
\|S\|\le \|S\|_{N},
\end{equation}
where $\|S\|$ is the norm of $S$ as a bilinear form.

Denote by $\mathcal{N}(\mathfrak{g}_1,\mathfrak{g}_2;\mathbb{K})$ the space of nuclear bilinear forms with the nuclear norm. It was shown in \cite{RYAN} that $\mathcal{N}(\mathfrak{g}_1,\mathfrak{g}_2;\mathbb{K})$ becomes a Banach space, and  
\begin{equation*}
\mathcal{N}(\mathfrak{g}_1,\mathfrak{g}_2;\mathbb{K})\simeq \mathfrak{g}^{\ast}_{1}\hat{\otimes}_{\pi} \mathfrak{g}^{\ast}_2/ \Ker(J)
\end{equation*}
as Banach spaces.

Similarly, using the natural algebraic embedding from $\mathfrak{g}_1 \otimes \mathfrak{g}_2$ to $\mathcal{L}(\mathfrak{g}_1^{\ast}, \mathfrak{g}_2^{\ast};\mathbb{K})$, we get an operator
\begin{equation*}
P:\mathfrak{g}_1 \hat{\otimes}_{\pi} \mathfrak{g}_2 
 \to \mathcal{L}(\mathfrak{g}^{\ast}_1 , \mathfrak{g}^{\ast}_2; \mathbb{K}).
\end{equation*}
 As $\mathfrak{g}_1 \subseteq\mathfrak{g}^{\ast\ast}_1$ and $\mathfrak{g}_2 \subseteq\mathfrak{g}^{\ast\ast}_2,$ one can check that $\im (P)\subseteq \mathcal{N}(\mathfrak{g}^{\ast}_1 , \mathfrak{g}^{\ast}_2; \mathbb{K})$ as linear spaces. Equipped with the nuclear norm, $\im(P)$ becomes a norm space which is denoted by $\overline{\mathcal{N}}(\mathfrak{g}^{\ast}_{1},\mathfrak{g}^{\ast}_{2};\mathbb{K}).$ It is not difficult to verify that $\overline{\mathcal{N}}(\mathfrak{g}^{\ast}_{1},\mathfrak{g}^{\ast}_{2};\mathbb{K})$ is also a Banach space. As Banach spaces, we have 
 \begin{equation}\label{BN}
\overline{\mathcal{N}}(\mathfrak{g}^{\ast}_{1},\mathfrak{g}^{\ast}_{2};\mathbb{K})\simeq \mathfrak{g}_{1}\hat{\otimes}_{\pi} \mathfrak{g}_{2}/\Ker(P).
 \end{equation}
 Therefore, there exists a natural surjective continuous map 
 \begin{equation}\label{GA}
p:\mathfrak{g}_{1}\hat{\otimes}_{\pi} \mathfrak{g}_{2}\to \overline{\mathcal{N}}(\mathfrak{g}^{\ast}_{1},\mathfrak{g}^{\ast}_{2};\mathbb{K}).
 \end{equation}
As $\overline{\mathcal{N}}(\mathfrak{g}^{\ast}_{1},\mathfrak{g}^{\ast}_{2};\mathbb{K})\subseteq \mathcal{L}(\mathfrak{g}^{\ast}_{1},\mathfrak{g}^{\ast}_{2};\mathbb{K})$ as linear spaces, then by \eqref{CP}, $p$ induces a continuous map from $\frkg_{1}\otp\frkg_{2}$ to $\mathcal{L}(\mathfrak{g}^{\ast}_{1},\mathfrak{g}^{\ast}_{2};\mathbb{K})$. We still denote this map by $p.$

\subsubsection{Skew-symmetric and symmetric projective tensor products}

Let $\mathfrak{g}$ be a Banach space. Define the map $\tau:\mathfrak{g}\otimes\mathfrak{g}\to \mathfrak{g}\otimes \mathfrak{g}$ by 
$$\tau(x\otimes y)=y\otimes x,\quad  \forall x,y\in \mathfrak{g}.$$ 
And let $\hat{\tau}:\mathfrak{g}\hat{\otimes}_{\pi} \mathfrak{g}\to \mathfrak{g}\hat{\otimes}_{\pi}\mathfrak{g}$ be the unique extension of $\tau$. An element $r\in \mathfrak{g}\hat{\otimes}_{\pi}\mathfrak{g}$ is called skew-symmetric if $\hat{\tau}(r)=-r$ and it is called symmetric if $\hat{\tau}(r)=r.$



\subsection{Duality pairings between Banach spaces and Banach Poisson spaces}\label{22}

\subsubsection{Banach Lie algebras}

Let $\mathfrak{g}$ be a Banach space and a Lie algebra with the Lie bracket $[\cdot,\cdot]$. Then $\mathfrak{g}$ is called a Banach Lie algebra if the Lie bracket $[\cdot,\cdot]$ is continuous, that is, there exists a constant $C>0$ such that 
\begin{equation}\label{BLA}
\big{\|}[x,y]\big{\|}\le C\|x\|\ \|y\| 
\end{equation}
for any $x,y\in \mathfrak{g}$.

\begin{rmk}
In some references on Banach Lie algebra, the condition $\eqref{BLA}$ is replaced by $\big{\|}[x,y]\big{\|}\le 2\|x\|\ \|y\|$ or $\big{\|}[x,y]\big{\|}\le \|x\|\ \|y\|$.
\end{rmk}

\subsubsection{Adjoint and coadjoint actions on Banach Lie algebras}

Here we use the terminology given in \cite{TP}. Let $\mathfrak{g}$ be a Banach Lie algebra. Recall that $\mathfrak{g}$ acts on itself and its continuous dual $\mathfrak{g}^{\ast}$:
$$
\begin{array}{llll}
\ad~: &\mathfrak{g}\times\mathfrak{g}&\longrightarrow &\mathfrak{g}\\
& (x, y) & \longmapsto & \ad_x y = [x , y ],
\end{array}
$$
$$
\begin{array}{llll}
-\ad^*~: &\mathfrak{g}\times\mathfrak{g}^*&\longrightarrow &\mathfrak{g}^*\\
& (x, \alpha) & \longmapsto & -\ad^*_x \alpha = -\alpha\circ \ad_x,
\end{array}
$$
where $\ad^{\ast}$ denotes the dual map of $\ad.$

\subsubsection{Duality pairings}

Let $\mathfrak{g}_1$ and $\mathfrak{g}_2$ be two normed spaces over the field $\mathbb{K},$ and let 
\begin{equation*}\label{DP}
	\langle \cdot,\cdot \rangle_{\mathfrak{g}_1,\mathfrak{g}_2}: \mathfrak{g}_1\times \mathfrak{g}_2 \to \mathbb{K}
\end{equation*}
be a continuous map. Then the map $\langle \cdot,\cdot \rangle_{\mathfrak{g}_1,\mathfrak{g}_2}$ is called a duality pairing between $\mathfrak{g}_1$ and $\mathfrak{g}_2$ if it is non-degenerate, that is, if the following two conditions hold:
\begin{itemize}
	\item[(a)]  $\langle x,y \rangle_{\mathfrak{g}_1, \mathfrak{g}_2}=0$ for any $y\in \mathfrak{g}_2$ implies $x=0$;
	\item[(b)] $\langle x,y \rangle_{\mathfrak{g}_1, \mathfrak{g}_2}=0$ for any $x\in \mathfrak{g}_1$ implies $y=0$.
\end{itemize}

Moreover, if the following two maps 
\begin{equation}\label{duality_pairing_strong}
	\begin{array}{lrl} \mathfrak{g}_1& \longrightarrow & \mathfrak{g}_2^*\\ x & \longmapsto &  \langle x,\cdot\rangle_{\mathfrak{g}_1, \mathfrak{g}_2}\end{array}
	\quad
	\textrm{and}\quad
	\begin{array}{lrl} \mathfrak{g}_2& \longrightarrow & \mathfrak{g}_1^*\\ y & \longmapsto &  \langle \cdot,y\rangle_{\mathfrak{g}_1, \mathfrak{g}_2}\end{array}
\end{equation} are isomorphic, then $\langle \cdot,\cdot\rangle_{\mathfrak{g}_1,\mathfrak{g}_2}$ is called a strong duality pairing between $\mathfrak{g}_1$ and $\mathfrak{g}_2.$

\subsubsection{Banach Poisson-Lie spaces}

Let $(\mathfrak{g}_+,\mathfrak{g}_-)$ be a duality pairing between Banach Lie algebras. The space $\mathfrak{g}_-$ is called a Banach Poisson-Lie space with respect to $\mathfrak{g}_+$ if $\mathfrak{g}_+$ acts continuously on $\mathfrak{g}_-$ by the coadjoint action, that is,
$\mathfrak{g}_-$ is preserved by the coadjoint action and the action map $$\ad^{\ast}:\mathfrak{g}_+\times \mathfrak{g}_- \to \mathfrak{g}_-$$ is continuous.

\subsubsection{Adjoint action on the space of continuous multilinear maps and the projective tensor product Banach spaces}

Here we use the terminology introduced in \cite{TP}. Let $\mathfrak{g}_{+}$ and $\mathfrak{g}_{-}$ be Banach Lie algebras with duality pairing $\langle\cdot ,\cdot\rangle$, such that $\mathfrak{g}_-$ is a Banach Poisson-Lie space with respect to $\mathfrak{g}_+$. In the paper, we write $\ad^{\ast}$ for $\ad^{\ast}_{\mathfrak{g}_{-}}$. To save space, we denote $\mathcal{L}(\underbrace{\mathfrak{g}_{-},\ldots,\mathfrak{g}_{-}}_{k};\mathbb{K})$ by $L^{k}(\mathfrak{g}_{-},\mathbb{K})$, which is the Banach spaces of continuous multilinear maps from $\mathfrak{g}_{-}\times\cdots \times \mathfrak{g}_{-}$ to $\mathbb{K}$, where $\mathfrak{g}_{-}$ are repeated $k$ times. Then we define the continuous action of $\mathfrak{g}_{+}$ on $L^{k}(\mathfrak{g}_{-};\mathbb{K})$, called the adjoint action, by 
\begin{align*}
	\ad_x f (y_1, \dots, y_k)= \sum_{i =1}^{k} f(y_1, \dots, \ad^*_x y_i, \dots, y_k),
\end{align*}
where $f\in L^{k}(\mathfrak{g}_{-};\mathbb{K})$ and $y_1, \dots, y_k\in \mathfrak{g}_{-}$.

Similarly, we can define the adjoint action on projective tensor product Banach spaces. Let $\frkg$ be a Banach space, and denote the space $\frkg\otimes\cdots\otimes \frkg$ with $\frkg$ repeated $k$ times by $\frkg^{\otimes k}.$ similarly, denote $\frkg\otp\cdots\otp\frkg$ with $\frkg$ repeated $k$ times by $\frkg^{\otp k}.$ For any $x\in \frkg$, define $\ad_{x}:\frkg^{\otimes k}\to \frkg^{\otimes k}$ by 
$$\ad_{x}(x_{1}\otimes x_{2}\cdots\otimes x_{k})=\sum_{i=1}^{k} x_{1}\otimes\cdots \otimes\ad_{x}(x_{i})\otimes \cdots\otimes x_{k},$$ 
where $x_1,\ldots,x_{k}\in \frkg.$ Extending $\ad_{x}$ with respect to the projective tensor norm, we get a unique operator from $\frkg^{\otp k}$ to $\frkg^{\otp k}.$ By abuse of
notation, we will also denote this operator by $\ad_x.$

Now we define the 1-coboundary on $\frkg^{\otimes k}.$ For any $r\in \frkg^{\otimes k}$, define $dr:\frkg\to \frkg^{\otimes k}$ by
$$dr(x)=\ad_{x} r,\quad \forall x\in \frkg.$$ 
And it is similar to define the 1-coboundary on $\frkg^{\otp k}.$

\subsubsection{Skew-symmetric and symmetric continuous bilinear maps}

For a Banach space $\mathfrak{g}$, the set of skew-symmetric continuous bilinear maps on $\mathfrak{g}$ is given by
$$\Lambda^2\mathfrak{g}^{\ast} = \left\{S\in \mathcal{L}(\mathfrak{g}, \mathfrak{g};\mathbb{K})~:~\forall x, y\in \mathfrak{g}, S(x, y) = -S(y,x)\right\}.$$
Also, an element $S\in \mathcal{L}(\mathfrak{g}, \mathfrak{g};\mathbb{K})$ is called symmetric if for all $x,y\in\mathfrak{g}$, we have $S(x,y)=S(y,x)$. For a nuclear map $S\in \mathcal{N}(\mathfrak{g}^{\ast},\mathfrak{g}^{\ast};\mathbb{K})$, $S$ is called skew-symmetric (resp. symmetric) if it is skew-symmetric (resp. symmetric) as an element of $\mathcal{L}(\mathfrak{g}^{\ast}, \mathfrak{g}^{\ast};\mathbb{K})$.

Let $\langle \cdot,\cdot \rangle_{\mathfrak{g}_+,\mathfrak{g}_-}$ be a duality pairing between Banach spaces $\mathfrak{g}_+$ and $\mathfrak{g}_
-.$  By the definition, $\mathfrak{g}_+\subseteq \mathfrak{g}^{\ast}_-$. Denote by $\Lambda^2 \mathfrak{g}_+$ the subspace of $\Lambda^2\mathfrak{g}^{\ast}_-$ consisting of $S\in\Lambda^{2} \mathfrak{g}^{\ast}_-$ such that for any $x\in \mathfrak{g}_-,$ the map $S_{x}\in \mathfrak{g}_+,$ where $S_{x}:\mathfrak{g}_-\to \mathbb{K}$ is given by 
$$S_{x}(y)=S(x,y),\quad \forall y\in \mathfrak{g}_-.$$ 
That is,
$$\Lambda^{2}\mathfrak{g}_+ =\left\{S\in \Lambda^2 \mathfrak{g}^{\ast}_-:\ \forall x \in \mathfrak{g}_-,\ S_{x}\in \mathfrak{g}_+  \right\}.$$

\subsection{Banach Lie algebras and projective Banach Lie biagebras}\label{23}

In \cite{TP}, as the differentiations of Banach Poisson Lie groups, the notion of Banach Lie bialgebra was introduced. Considering Banach Lie bialgebras in the context of projective tensor products, we give the notion of projective Banach Lie bialgebras. 

\subsubsection{Banach Lie bialgebras}

Let $\langle \cdot,\cdot\rangle_{\mathfrak{g_+},\mathfrak{g_-}}$ be a duality pairing between Banach spaces $\mathfrak{g}_+$ and $\mathfrak{g}_-$ with $\mathfrak{g}_+$ a Banach Lie algebra. Then $\mathfrak{g}_+$ is called a \textbf{Banach Lie bialgebra} with respect to $\mathfrak{g}_-$ if 
\begin{itemize}
    \item[(a)]  $\mathfrak{g}_+$ acts continuously by coadjoint action on $\mathfrak{g}_-$, that is, $\ad^{\ast}_{\mathfrak{g_{+
    }}} \mathfrak{g}_{-}\subseteq \mathfrak{g}_{-}$ and the map 
    $$
    \begin{aligned}
    &\mathfrak{g}_{+}\times \mathfrak{g}_{-}\to \mathfrak{g}_{-}\\
    &(x_{+},x_{-})\mapsto \ad^{\ast}_{x_{+}} x_{-} 
    \end{aligned} 
    $$ is continuous.
    \item[(b)]  There is a Lie bracket $[\cdot,\cdot]_{\mathfrak{g}_-}:\mathfrak{g}_-\times \mathfrak{g}_-\to\mathfrak{g}_-$ such that $\mathfrak{g}_-$ is a Banach Lie algebra and $[\cdot, \cdot]_{\mathfrak{g}_-}~:\mathfrak{g}_-\times\mathfrak{g}_-\rightarrow\mathfrak{g}_-$ restricts to
  a $1$-cocycle $\theta~: \mathfrak{g}_+\rightarrow \Lambda^2\mathfrak{g}_-^*$ with respect to the adjoint action $\ad$ of $\mathfrak{g}_+$ on  $\Lambda^2\mathfrak{g}_-^*$, that is, for any $x, y\in \mathfrak{g}_+$,
  \begin{equation*}\begin{array}{ll}
\theta\left([x, y]\right) & = \ad_x\left(\theta(y)\right) - \ad_y\left(\theta(x)\right).
\end{array}
\end{equation*}
\end{itemize}

\subsubsection{Projective Banach Lie bialgebras}

\begin{defi} 
Let $\mathfrak{g}$ be a Banach Lie algebra and $\delta:\mathfrak{g}\to\mathfrak{g}\hat{\otimes}_{\pi}\mathfrak{g}$ be a bounded operator. Then the pair
$(\mathfrak{g},\delta)$ is called a \textbf{projective Banach Lie bialgebra} if it satisfies:
\begin{itemize}
\item[(a)] the operator $\delta^{\ast}|_{\mathfrak{g}^{\ast}\otimes \mathfrak{g}^{\ast}}:\frkg^{\ast}\times\frkg^{\ast}\to \frkg^{\ast}$ given by 
\begin{equation}\label{PBL}
\delta^{\ast}|_{\mathfrak{g}^{\ast}\otimes \mathfrak{g}^{\ast}}(\alpha,\beta)(x)=(\alpha\hat{\otimes}_{\pi} \beta)\delta(x)\quad\forall \alpha,\beta\in \mathfrak{g}^{\ast},\quad\forall x\in \mathfrak{g}
\end{equation}
is a Lie bracket, where $\delta^{\ast}$ and $\alpha \otp \beta$ denote the transpose of $\delta$ and the unique extension of $\alpha\otimes \beta$ respectively;
\item[(b)] $\delta$ is a 1-cocycle, that is, for any $x,y\in \frkg$ 
\begin{equation*}
\delta([x,y])=\ad_{x}(\delta(y))-\ad_{y}(\delta(x)).
\end{equation*}
\end{itemize}
\end{defi}

\begin{rmk}
\begin{itemize}
\item[(a)] By $\eqref{GA},$ for any $x\in\frkg,$ $\delta(x)$ induces a unique continuous bilinear map lies in $\overline{\mathcal{N}}(\mathfrak{g}^{\ast},\mathfrak{g}^{\ast};\mathbb{K})$. Hence $\eqref{PBL}$ also means that 
$$\delta^{\ast}|_{\frkg^{\ast}\otimes\frkg^{\ast}}(\alpha,\beta)(x)=p(\delta(x))(\alpha,\beta),$$ 
where $p$ is given in $\eqref{GA}.$ 
\item[(b)] In the above definition, as $\delta$ is bounded, we find that $$\|\delta^{\ast}|_{\mathfrak{g}^{\ast}\otimes \mathfrak{g}^{\ast}}(\alpha,\beta)\|\le  \|\delta\|\ \|\alpha\|\ \|\beta\|.$$ 
Hence condition (a) is equivalent to say that $\frkg^{\ast}$ is a Banach Lie algebra with repsect to $\delta^{\ast}|_{\mathfrak{g}^{\ast}\otimes \mathfrak{g}^{\ast}}.$
\end{itemize}
\end{rmk}

In the next proposition, we prove that every projective Banach Lie bialgebra has the Banach Lie algebra structure.

\begin{pro}
Let $(\mathfrak{g},\delta)$ be a projective  Banach Lie bialgebra and $[\cdot,\cdot]_{\mathfrak{g}^{\ast}}$
be the restriction of $\delta^{\ast}$ on $\mathfrak{g}^{\ast}\otimes \mathfrak{g}^{\ast}$. Then $\frkg^{\ast}$ is a Lie algebra with 
Lie bracket $[\cdot,\cdot]_{\mathfrak{g}^{\ast}}$ and $\mathfrak{g}$ is a Banach Lie bialgebra with respect to $\mathfrak{g}^{\ast}.$
\end{pro}

\begin{proof}
 By the definition, it is easy to see that the bilinear map $\ad^{\ast}:\frkg \times\frkg^{\ast}\to \frkg^{\ast}$ is continuous.  Denote by $\theta$ the restriction of $[\cdot,\cdot]_{\frkg^{\ast}}^{\ast}$ on $\frkg.$ It is enough to show that $\theta$ is a 1-cocycle.
By Proposition \ref{UPTP}, for any $x,y\in\mathfrak{g}$ and $\alpha,\beta\in \mathfrak{g}^{\ast}$, we have 
$$\begin{aligned}
\langle [\alpha, \beta]_{\mathfrak{g}^{\ast}},[x,y]\rangle- \langle [\ad_{x}^{\ast} \alpha,\beta]_{\mathfrak{g}^{\ast}},y\rangle-\langle [\alpha,\ad_{x}^{\ast} \beta ]_{\mathfrak{g}^{\ast}},y\rangle+\langle [\ad_{y}^{\ast} \alpha,\beta]_{\mathfrak{g}^{\ast}},x\rangle+\langle [\alpha,\ad_{y}^{\ast} \beta]_{\mathfrak{g}^{\ast}},x\rangle=0.
\end{aligned}
$$
Then by the definition we have 
$$\begin{aligned}
\langle \alpha\otimes\beta,\theta([x,y])\rangle- \langle \ad_{x}^{\ast} \alpha\otimes\beta,\theta(y)\rangle-\langle \alpha\otimes\ad_{x}^{\ast} \beta ,\theta(y)\rangle+\langle \ad_{y}^{\ast} \alpha\otimes\beta,\theta(x)\rangle+\langle \alpha\otimes\ad_{y}^{\ast} \beta,\theta(x)\rangle=0.
\end{aligned}
$$
Finally, we get 
\begin{equation*}\begin{array}{ll}
\theta\left([x, y]\right) -\ad_x\left(\theta(y)\right) +\ad_y\left(\theta(x)\right)=0,
\end{array}
\end{equation*}
which means that $\theta$ is a 1-cocycle.
\end{proof}

\subsubsection{Projective Banach Lie bialgebras and concrete Banach Lie bialgebras}

Here we propose the notion of projective Lie coalgebras, which coincides with topological Lie coalgebras given in \cite{MX2024} and completed Lie coalgebras proposed in \cite{B2024}.

\begin{defi}\label{LC}
A projective Banach Lie coalgebra is a pair $(\frkg,\delta)$ where $\mathfrak{g}$ is a Banach space and $\delta:\frkg\to \frkg\otp\frkg$ is a bounded linear map satisfying
\begin{itemize}
\item[(a)] $\delta(x)+\hat{\tau}\delta(x)=0$;
\item[(b)] $\hat{\Alt}((\delta\otp\id)(\delta(x)))=0$
\end{itemize}
for any $x\in \frkg.$ Here $\Alt:\frkg\otimes \frkg \otimes \frkg\to \frkg\otimes \frkg\otimes \frkg$ is given by 
\begin{equation*}
\Alt(x\otimes y\otimes z)=x\otimes y\otimes z+y\otimes z\otimes x+z\otimes x\otimes y,\quad \forall x,y,z\in \frkg,
\end{equation*}
and $\hat{\tau},\hat{\Alt},\delta\otp \id$ stand for the unique continuous extension of $\tau$, $\Alt$ and $\delta\otp\id$ respectively.
\end{defi}

Next we introduce the notion of concrete Banach Lie bialgebras.

\begin{defi}
 A pair $(\frkg,\delta)$ is called a \textbf{concrete Banach Lie bialgebra} if it satisfies the following conditions:
\begin{itemize}
\item[(a)] $\frkg$ is a Banach Lie algebra;
\item[(b)] $(\frkg,\delta)$ is a projective Banach Lie coalgebra;
\item[(c)]
$\delta$ is a $1$-cocycle, that is, 
\begin{equation*}
\delta([x,y])=\ad_{x}(\delta(y))-\ad_{y}(\delta(x)),\quad \forall x,y\in \frkg.
\end{equation*}
\end{itemize}
\end{defi}

\begin{pro}\label{Prop:concrete-projective}
Let $(\frkg,\delta)$ be a concrete Banach Lie bialgebra. Then $(\frkg,\delta)$ is a projective Banach Lie bialgebra.
\end{pro}

\begin{proof}
By Definition \ref{LC}, one can readily verify that $\delta^{\ast}|_{\frkg^{\ast}\otimes \frkg^{\ast}}$ is a Lie bracket on $\frkg^{\ast}$. Hence $(\frkg,\delta)$
is a projective Banach Lie bialgebra.
\end{proof}

\begin{rmk}
By \cite{RYAN}, if $\frkg$ is of finite dimension, then $\frkg\otimes_{\pi} \frkg$ is complete. It is well-known that projective Banach Lie bialgebras and concrete Banach Lie bialgebras are equivalent structures in the finite-dimensional setting. However, this equivalence breaks down in the infinite-dimensional case. As demonstrated by Proposition \ref{Prop:concrete-projective}, every concrete Banach Lie bialgebra has the structure of projective Banach Lie bialgebras. While the converse is not true. This is primarily due to the inequivalence of spaces $\overline{\mathcal{N}}(\mathfrak{g}^{\ast},\mathfrak{g}^{\ast};\mathbb{K})$ and $\frkg\otp \frkg$  in the infinite-dimensional context. By \eqref{BN}, we know that the two spaces are not necessarily isomorphic . 

Though concrete Banach Lie bialgebras appear to more naturally align with the classical definition of Lie bialgebras, our investigation in this paper centers primarily on projective Banach Lie bialgebras, motivated by two key considerations. Firstly, projective Banach Lie bialgebras are better suited for investigating the relationship with the projective Yang-Baxter equation given in Section \ref{3}. Secondly, projective Banach Lie bialgebras provide a more favorable framework for studying their corresponding Lie groups structure.  
\end{rmk}

\subsection{Banach Manin triples and projective Banach Lie bialgebras}\label{24}

Finally, we recall the concepts on Banach Manin triples and investigate their connection with projective Banach Lie bialgebras.

Let $\mathfrak{g}$ be a Banach Lie algebra and $S:\mathfrak{g}\times\mathfrak{g}\to \mathfrak{g}$ be a bilinear form. Then $S$ is called continuous if it is continuous as a bilinear map, and $S$ is called invariant if it satisfies 
$$S(x,[y,z])=S([x,y],z),\quad\forall x,y,z\in \mathfrak{g}.$$ 
The bilinear form $S$ is called non-degenerate if $S(x,y)=0$ for any $y\in\mathfrak{g}$ implies $x=0$ and $S(x,y)=0$ for any $x\in \mathfrak{g}$ implies $y=0$. 

Let $\mathfrak{g}$ be a vector space with a bilinear form $S$ and $\mathfrak{g}_1$ be a subspace of $\mathfrak{g}$. Define 
$$\mathfrak{g}_1^{\bot}=\big\{x\in\mathfrak{g}\ :\ S(x,\mathfrak{g})=0\big\}.$$
Recall that the subspace $\mathfrak{g}_1$ is called \textbf{isotropic} if $\mathfrak{g}_1\subseteq \mathfrak{g}_1^{\bot}.$ And $\mathfrak{g}_1$ is called \textbf{Lagrangian} if $\mathfrak{g}_1=\mathfrak{g}_1^{\bot}.$ 

\begin{defi}
Let $\mathfrak{g},\mathfrak{g}_+$ and $\mathfrak{g}_-$ be Banach Lie algebras and $S$ be a continuous bilinear form on $\mathfrak{g}$. The triple $((\mathfrak{g},S),\mathfrak{g}_+,\mathfrak{g}_-)$ is called a Banach Manin triple if it satisfies:

\begin{itemize}
\item[(a)] $S$ is non-degenerate and invariant;
\item[(b)] $\mathfrak{g}=\mathfrak{g}_{+}\oplus\mathfrak{g}_{-}$ as Banach spaces;
\item[(c)] both $\mathfrak{g}_+$ and $\mathfrak{g}_-$ are Banach Lie subalgebras of $\mathfrak{g}$;
\item[(d)] $\mathfrak{g}_+$ and $\mathfrak{g}_-$ are isotropic with respect to $S$.
\end{itemize}
\end{defi}

Let $(\mathfrak{g},\delta)$ be a projective Banach Lie bialgebra. Denote by $\mathfrak{g}\dot{+}\mathfrak{g}^{\ast}$ the direct sum of $\mathfrak{g}$ and $\mathfrak{g}^{\ast}$ as Banach spaces but not necessarily as Lie algebras. Equip $\mathfrak{g}\dot{+}\mathfrak{g}^{\ast}$ with a bilinear form 
$$S(x+\alpha,y+\beta)=\alpha(y)+\beta(x).$$
Then one can verify  
\begin{equation}\label{PD}
[x,\alpha]=-\alpha\circ \ad_{x}+(\alpha\hat{\otimes}_{\pi} \id)\delta(x)
\end{equation}
is the unique Lie bracket on $\mathfrak{g}\dot{+}\mathfrak{g}^{\ast}$ making $((\mathfrak{g},S),\mathfrak{g},\mathfrak{g}^{\ast})$ into a Banach Manin triple. 

 The converse is not true. Not every Banach Manin triple $((\mathfrak{g},S),\mathfrak{g},\mathfrak{g}^{\ast})$ on Banach Lie algebras $\frkg$ and $\frkg^{\ast}$ gives rise to a projective Banach Lie bialgebra structure by the definition. However, if the dual map of the Lie bracket of $\mathfrak{g}^{\ast}$
 $$[\cdot,\cdot]^{\ast}:\frkg^{\ast\ast}\to (\mathfrak{g}^{\ast}\otimes\mathfrak{g}^{\ast})^{\ast}$$
 restricts to the map $\delta:\mathfrak{g}\to \mathfrak{g}\hat{\otimes}_{\pi}\mathfrak{g}$, here we use the inclusion $\mathfrak{g}^{\ast}\otimes\mathfrak{g}^{\ast}\subseteq\mathcal{L}(\mathfrak{g},\mathfrak{g};\mathbb{K})$ and Proposition \ref{UPTP}, then $(\mathfrak{g},\delta)$ is a Banach Lie bialgebra. In other words, if there is a bounded linear map $\delta:\mathfrak{g}\to \mathfrak{g}\hat{\otimes}_{\pi} \mathfrak{g}$ satisfying the condition 
 $$S(\alpha\otimes\beta,\delta(x))=S([\alpha,\beta],x),\quad \forall x\in \mathfrak{g},\ \forall \alpha,\beta\in \mathfrak{g}^{\ast}, $$ then $(\mathfrak{g},\delta)$ is a Banach Lie bialgebra.

\section{Quasi-triangular projective Banach Lie bialgebras}\label{3}

In Section \ref{31}, we recall the definition of the completion of the classical Yang-Baxter equation in the projective tensor norm, called the projective Yang-Baxter equation. This kind of equation first appeared in \cite{MX2024}. A solution of the projective Yang-Baxter equation is called a projective r-matrix. Then in Section \ref{3.2}, we prove that every projective r-matrix gives rise to a projective Banach Lie bialgebra structure.

\subsection{1-coboundaries of projective Banach Lie algebras and the completion of the classical Yang-Baxter equation}\label{31}

First, we introduce the notion of 1-coboundaries on projective Banach Lie algebras as the projective Banach Lie algebra analogue of 1-coboundaries on Lie algebras.

\begin{defi}\label{1C}
Let $\mathfrak{g}$ be a Banach Lie algebra. A projective 1-coboundary with respect to the adjoint action $\ad$ of $\mathfrak{g}$ on $\frkg\otp\frkg$ is a map $\delta:\mathfrak{g}\to \mathfrak{g}\hat{\otimes}_{\pi} \mathfrak{g}$ such that there is an element $r\in \mathfrak{g}\hat{\otimes}_{\pi}\mathfrak{g}$ satisfying 
\begin{equation}\label{DLT}
\delta(x)=dr(x)=\ad_{x} r=(\ad_{x}\hat{\otimes}_{\pi}\id+\id\hat{\otimes}_{\pi} \ad_{x})r,\quad  \forall x\in\mathfrak{g}.
\end{equation}
\end{defi}

Let $\mathfrak{g}$ be a Banach algebra and $\delta=dr$ be a projective 1-coboundary on $\mathfrak{g}$. Let $\delta^{\ast}$ be the transpose of $\delta$. To obtain a projective  Banach Lie algebra structure on $\mathfrak{g}^{\ast}$ such that the Lie bracket is given by $\delta^{\ast}$, it remains to verify:
\begin{itemize}
    \item[(a)] For any $\alpha,\beta\in \mathfrak{g}^{\ast}$, there is a $C>0$ such that
    \begin{equation}\label{TOD1}
    ||\delta^{\ast}(\alpha,\beta)||\le C||\alpha||\;||\beta||;
    \end{equation}
    \item[(b)] $dr(x)$ is skew-symmetric for any $x\in\mathfrak{g}$;
    \item[(c)] The restriction of ${\delta}^{\ast}$ on $\mathfrak{g}^{\ast}\otimes\mathfrak{g}^{\ast}$ satisfies the Jacobi identity. 
\end{itemize}
It follows from \eqref{PBL} that condition (a) holds for any projective 1-coboundary. 

In \cite{MX2024}, the topological classical Yang-Baxter equation was introduced. As the topological classical Yang-Baxter equation is endowed with the projective tensor norm, next we define the completion of the classical Yang-Baxter equation.

Let $\mathfrak{g}$ be a Banach Lie algebra and $r=\sum_{i} u_i\otimes v_i\in \mathfrak{g}\otimes \mathfrak{g}.$ In $\mathfrak{g}\otimes \mathfrak{g}\otimes\mathfrak{g}$, let $r_{12},$ $r_{13}$ and $r_{23}$ be elements in the third tensor power of the enveloping algebra of $\mathfrak{g},$ defined respectively by
\begin{equation}\label{r1}
r_{12}=\sum_{i} u_i\otimes v_i\otimes 1,
\end{equation}
\begin{equation}\label{r2}
r_{13}=\sum_{i} u_i\otimes 1\otimes v_i
\end{equation}
and 
\begin{equation}\label{r3}
r_{23}=\sum_{i}   1\otimes u_i\otimes v_i.
\end{equation}
Define
\begin{equation*}
\begin{aligned}
&[r_{12},r_{13}]=\left[\sum_{i} u_i\otimes v_i\otimes 1,\sum_{j} u_j\otimes 1\otimes v_j\right]=\sum_{i,j}[u_i,u_j]\otimes v_i\otimes v_j,\\
&[r_{12},r_{23}]=\left[\sum_{i} u_i\otimes v_i\otimes 1,\sum_{j} 1\otimes u_j\otimes v_j\right]=\sum_{i,j}u_i\otimes [v_i,u_j]\otimes v_j,\\
&[r_{13},r_{23}]=\left[\sum_{i} u_i\otimes 1\otimes v_i,\sum_{j} 1\otimes u_j\otimes v_j\right]=\sum_{i,j} u_i\otimes
u_j\otimes [v_i,v_j].
\end{aligned}
\end{equation*}
Denote $[r_{12},r_{13}]+[r_{13},r_{23}]+[r_{12},r_{23}]$ by $\CYB (r)$ or $\langle r,r\rangle.$
Then $r$ is called a solution of \textbf{the classical Yang-Baxter equation} if it satisfies
$$\CYB(r)=\langle r,r \rangle=0.$$ And $\CYB(r)=0$ is called the classical Yang-Baxter equation.

Now we show that $\CYB$ is a continuous map from $\mathfrak{g}\otimes \mathfrak{g}$ to $\mathfrak{g}\otimes \mathfrak{g}\otimes \mathfrak{g}$ in the projective norm.

\begin{lem}\label{CYBC}
Let $\mathfrak{g}$ be a Banach Lie algebra. Then the operator $\CYB:\mathfrak{g}\otimes \mathfrak{g}\to \mathfrak{g}\otimes\mathfrak{g}\otimes\mathfrak{g}$ is continuous in the projective norm.
\end{lem}

\begin{proof}
As $\mathfrak{g}$ is a Banach Lie algebra, there exists a $C>0$ such that 
$$\|[x,y]\|\le C\|x\|\ \|y\|,\quad \forall x,y\in\mathfrak{g}.$$ 
For any $r\in\mathfrak{g}\otimes \mathfrak{g}$ with representation $r=\sum_{i} u_i\otimes v_i$, using the above notations, we have 
$$
\begin{aligned}
\|\CYB(r)\|_{\pi}=&\left\|\sum_{i,j}\left( [u_i,u_j]\otimes v_i\otimes v_j+u_i\otimes [v_i,u_j]\otimes v_j+u_i\otimes u_j\otimes [v_i,v_j]\right)\right\|_{\pi}\\
\le& 3C\sum_{i,j} \left(\|u_i\|\ \|v_i\|\right)\left(\|u_j\|\ \|v_j\|\right)=3C\left(\sum_{i} \|u_i\|\ \|v_i\|\right)\left(\sum_{j}\|u_j\| \ \|v_j\|\right),
\end{aligned}
$$
which implies 
$$\|\CYB(r)\|_{\pi}\le 3C\|r\|^{2}_{\pi}.$$
Hence $\CYB$ is continuous in the projective norm. 
\end{proof}

Denote the unique extension of $\CYB$ from $\mathfrak{g}\hat{\otimes}_{\pi} \mathfrak{g}$ to $\mathfrak{g}\hat{\otimes}_{\pi}\mathfrak{g}\hat{\otimes}_{\pi} \mathfrak{g}$ by $\hat{\CYB}.$ An element $r\in \mathfrak{g}\hat{\otimes}_{\pi} \mathfrak{g}$ is called a solution of \textbf{the projective Yang-Baxter equation} if it satisfies 
$\hat{\CYB}(r)=0.$

Then we define the projective solutions of the classical Yang-Baxter equation.

\begin{defi}
Let $\mathfrak{g}$ be a Banach Lie algebra. With the above notations, an element $r\in \mathfrak{g}\hat{\otimes}_{\pi} \mathfrak{g}$ is called a \textbf{projective solution} of the classical Yang-Baxter equation if there exists a sequence $(r_n\in \mathfrak{g}\otimes_{\pi} \mathfrak{g})_{n=0}^{\infty}$, such that 
\begin{equation*}
\begin{array}{lrl}
{\lim\limits_{n\to +\infty}}r_n=r
\end{array}
\quad
\textrm{and}\quad
\begin{array}{lrl}
{\lim\limits_{n\to +\infty}\langle r_{n},r_{n}\rangle}=0
\end{array}
\end{equation*}
in the projective norm.
\end{defi}

Next, we give equivalent characterizations of the solutions of the projective Yang-Baxter equation.

\begin{pro}
Let $\mathfrak{g}$ be a Banach Lie algebra and $r$ be an element of $\mathfrak{g}\hat{\otimes}_{\pi} \mathfrak{g}$. Then the following assertions are equivalent:
\begin{itemize}
\item[{\rm(a)}] $r$ is a solution of the projective Yang-Baxter equation.
\item[{\rm(b)}] $r$ is a projective solution of the classical Yang-Baxter equation.
\item[{\rm(c)}] There exists bounded sequences $(x_k\in \frkg)$ and $(y_k\in \frkg)$ such that $\sum_{k=0}^{\infty} x_k\otimes y_k$ converges to $r$ 
and $\lim\limits_{n\to +\infty} \left\langle \sum_{i=0}^{n} x_{i}\otimes y_{i},\sum_{j=0}^{n} x_{j}\otimes y_{j}\right\rangle=0$ in the projective norm.
\end{itemize}
\end{pro}

\begin{proof}
By the definition and Lemma \ref{CYBC}, one can readily verify that (a) is equivalent to (b). It follows from \cite[Proposition 2.8]{RYAN} that (b) is equivalent to (c).
\end{proof}

Then let us give some examples of solutions of the projective Yang-Baxter equation.

\begin{ex}\label{E1}
Let $\mathcal{H}$ be a separable Hilbert space and $(e_k)_{k=0}^{\infty}$ be a orthonormal basis of $\mathcal{H}.$ Denote the set of bounded operators on $\mathcal{H}$ by $\mathcal{L}(\mathcal{H})$, then $\mathcal{L}(\mathcal{H})$ is a Banach Lie algebra with the Lie bracket given by the commutators.

And for any $i,j\in \mathbb{N},$ define the operator $E_{ij}:\mathcal{H}\to \mathcal{H}$ by 
$$E_{ij}(x)=\langle e_{i},x\rangle e_j,\quad\forall x\in \mathcal{H}.$$
Then for any $(a_i)_{i=0}^{\infty}\in \ell^{1},$
the operator $\sum_{i=0}^{\infty}  a_{i} (E_{ii}\otimes E_{(i+1)(i+1) }-E_{(i+1)(i+1)}\otimes E_{ii})$ is a projective solution of the classical Yang-Baxter equation.
\end{ex}

\begin{ex}\label{E2}
Let $\mathcal{H}$ be a separable Hilbert space and $(e_k)_{k=0}^{\infty}$ be a orthonormal basis of $\mathcal{H}.$ 
If the sequence $(a_{ij}\in \mathbb{K })_{0\le i<j}^{\infty}$ satisfies
\begin{itemize}
\item[(a)] $\sum_{0\le i<j}^{\infty} |a_{ij}|<\infty$;
\item[(b)] $a_{ij}a_{jk}=a_{ik}a_{ij} + a_{jk}a_{ik},$
\end{itemize}
then the operator $\sum_{0\le i<j}^{\infty}a_{ij} (E_{ij}\otimes E_{ji}-E_{ji}\otimes E_{ji})$ lies in $\mathcal{L}(\mathcal{H})$, and it is a skew-symmetric solution of the projective Yang-Baxter equation. For example, we can take $a_{ij}=\frac{1}{2^{j}-2^{i}}$.  
\end{ex}

\begin{ex}\label{E3}
Let $(\frkg_{i})_{i=1}^{\infty}$
be a family of Banach Lie algebras of finite dimension. And let $\frkg=\bigoplus_{i} \frkg_{i}$ be the set of all $(x_{i})\in \prod_{i} \frkg_{i}$ such that
\begin{equation}\label{IPN}
 \|(x_{i})\|=\sup_{i} \|x_{i}\|<\infty.
\end{equation}
Then $\frkg$ is a Banach Lie algebra such that as a Banach space, the norm is given by \eqref{IPN}, and as a Lie algebra, it is the direct sum of $(\frkg_{i}).$

Let $(\delta_{i}\in \frkg_{i}\otimes \frkg_{i})$ be a family of solutions of the projective Yang-Baxter equation and $(a_{i})_{i=1}^{\infty}\in \ell^{1}$. Then $\delta=\sum_{i} a_{i}\delta_{i}$ lies in $\frkg\otp \frkg$ and it is a solution of the projective Yang-Baxter equation. Moreover, if for any $i,$ $\delta_{i}$ is skew-symmetric, then $\delta$ is also skew-symmetric.
\end{ex}


\subsection{Quasi-triangular projective Banach Lie bialgebra}\label{3.2}

Let $\mathfrak{g}$ be a projective Banach Lie algebra and $a\in \mathfrak{g}\otimes\mathfrak{g}$ be skew-symmetric. Let us recall the definition of \textbf{the algebraic Schouten bracket} of $a$ given in \cite{K} and \cite{MX}. Here we use the Einstein notation, that is, for any $a\in \frkg\otimes\frkg,$ we write $a=\sum_{i}^{n} x_i\otimes y^{i}$ simply as $a=x_i\otimes y^{i}$. And denote the set of skew-symmetric elements of $\mathfrak{g}\otimes \mathfrak{g}$ by $\mathfrak{g}\land\mathfrak{g}$. The algebraic Schouten bracket of $a$ is the element
$$[\![a, a]\!]=-2\left([y^{i},y^{j}]\otimes x_i\otimes x_j+x_{j}\otimes [y^{i},y^{j}]\otimes x_{j}+x_{i}\otimes x_{j}\otimes [y^{i}, y^{j}]\right)\in \mathfrak{g}\land\mathfrak{g}\land\mathfrak{g}.$$ And for any $r\in \mathfrak{g}\otimes \mathfrak{g}$, define $\underline{r}:\mathfrak{g}^{\ast}\to \mathfrak{g}$ by $\underline{r}(\alpha)=(\alpha\otimes\id)(r)$ for any $\alpha\in\mathfrak{g}^{\ast}.$

One can readily verify that $[\![a,a]\!]$ is the unique element satisfying the following identity:
$$
\begin{aligned}
\langle\alpha\otimes \beta\otimes\gamma,[\![a,a]\!]\rangle&=-2\left(\langle \alpha,[\underline{a}(\beta),\underline{a}(\gamma)]\rangle+\langle \beta,[\underline{a}(\gamma),\underline{a}(\alpha)]\rangle+\langle \gamma,[\underline{a}(\alpha),\underline{a}(\beta)]\rangle\right)\\
&=-2\circlearrowleft \langle \alpha,[\underline{a}(\beta),\underline{a}(\gamma)]\rangle,
\end{aligned}
$$
for any $\alpha,\beta,\gamma\in \mathfrak{g}^{\ast}$, where $\circlearrowleft$ denotes the summation over the circular permutations of $\alpha,\beta,\gamma.$

It was shown in \cite{K} and \cite{MX} that 
\begin{equation}\label{JCB}
\langle \alpha\otimes\beta \otimes \gamma,d[\![a,a]\!](x)\rangle
=2\circlearrowleft\langle da^{\ast} (da^{\ast}(\alpha\otimes\beta)\otimes \gamma),x\rangle
\end{equation}
for any $\alpha,\beta,\gamma\in \mathfrak{g}^{\ast}$ and $x\in\mathfrak{g}.$

Moreover, by \cite{K,MX}, we have 
\begin{equation}\label{STCY}
\CYB(a)=\langle a,a \rangle=-\frac{1}{2}[\![a,a]\!].
\end{equation}

In \cite{K} and \cite[Proposition 1.1.5]{MX}, it was shown that the restriction of the linear map $da^{\ast}$ to $\mathfrak{g}\otimes \mathfrak{g}$ satisfies the Jacobi identity for any skew-symmetric element $a\in \mathfrak{g}\otimes \mathfrak{g}$. Similarly, we have the following result.

\begin{pro}\label{JK}
Let $\mathfrak{g}$ be a Banach Lie algebra and $a$ be a skew-symmetric element of $\mathfrak{g}\hat{\otimes}_{\pi}\mathfrak{g}.$ Then $da^{\ast}$ satisfies the Jacobi identity if and only if there exists a sequence $(a_n\in \mathfrak{g}\otimes \mathfrak{g})$ of skew-symmetric elements such that $\lim_{n\to +\infty} a_n=a$ and $\lim_{n\to +\infty} \ad_{x}[\![a_n,a_n]\!]=0$ for any $x\in \mathfrak{g}$ in the projective norm.
\end{pro}

\begin{proof}
Assume  that there exists a sequence of skew-symmetric elements $(a_n\in \mathfrak{g}\otimes \mathfrak{g})$ such that $\lim_{n\to +\infty} a_n=a$ and $\lim_{n\to +\infty} d[\![a_n,a_n]\!]=0$ in the projective norm. As for any $x\in \mathfrak{g}$, the operator $\ad_{x}$ is bounded on $\mathfrak{g}\hat{\otimes}_{\pi}\mathfrak{g}$, we have $\lim_{n\to\infty} da_n=da$ in the projective norm. It follows from Proposition \ref{UPTP}
that $\lim_{n\to \infty} da_{n}^{\ast}=da^{\ast}$ in the norm of $\mathcal{B}(\mathfrak{g}^{\ast},\mathfrak{g}^{\ast};\mathbb{K})$.

Then for any $\alpha,\beta,\gamma\in\mathfrak{g}^{\ast}$ and $x\in\mathfrak{g}$, we have
\begin{equation}\label{JCB1}
\begin{aligned}
&\langle da^{\ast} (da^{\ast}(\alpha\otimes\beta)\otimes \gamma),x\rangle-\langle da_n^{\ast} (da_n^{\ast}(\alpha\otimes\beta)\otimes \gamma),x\rangle\\
=&\left(\langle da^{\ast} (da^{\ast}(\alpha\otimes\beta)\otimes \gamma),x\rangle-\langle da_n^{\ast} (da^{\ast}(\alpha\otimes\beta)\otimes \gamma),x\rangle\right)\\&+\left(\langle da_n^{\ast} (da^{\ast}(\alpha\otimes\beta)\otimes \gamma),x\rangle-\langle da_n^{\ast} (da_n^{\ast}(\alpha\otimes\beta)\otimes \gamma),x\rangle\right)\\
=&\langle (da^{\ast}-da_{n}^{\ast}) (da^{\ast}(\alpha\otimes\beta)\otimes \gamma),x\rangle+\langle da_n^{\ast} ((da^{\ast}-da_n^{\ast})(\alpha\otimes\beta)\otimes \gamma),x\rangle. 
\end{aligned}
\end{equation}
Taking the limit and by \eqref{JCB}, we have 
\begin{align*}
\circlearrowleft\langle da^{\ast} (da^{\ast}(\alpha\otimes\beta)\otimes \gamma),x\rangle=&\lim_{n\to\infty} \circlearrowleft\langle  da_n^{\ast} (da_n^{\ast}(\alpha\otimes\beta)\otimes \gamma),x\rangle\\
=&\lim_{n\to\infty}\frac{1}{2}\langle \alpha\otimes\beta \otimes \gamma,d[\![a_n,a_n]\!](x)\rangle=0.
\end{align*}

Conversely, if $da^{\ast}$ satisfies the Jacobi identity, 
by definition, there exists a sequence of skew-symmetric elements $(a_n\in \mathfrak{g}\otimes\mathfrak{g})$ such that $\lim_{n\to+\infty} a_n=a$ in the projective norm. Then it follows from \eqref{JCB} and \eqref{JCB1} that 
\begin{align*}
\lim_{n\to\infty}\langle \alpha\otimes\beta \otimes \gamma,d[\![a_n,a_n]\!](x)\rangle=&2\lim_{n\to\infty} \circlearrowleft\langle  da_n^{\ast} (da_n^{\ast}(\alpha\otimes\beta)\otimes \gamma),x\rangle\\
=&2\circlearrowleft\langle  da^{\ast} (da^{\ast}(\alpha\otimes\beta)\otimes \gamma),x\rangle=0,
\end{align*}
which finishes the proof.
\end{proof}

Let $\mathfrak{g}$ be a Banach algebra and $r\in\mathfrak{g}\hat{\otimes}_{\pi}\mathfrak{g}$. The element $r$ is called $\ad$-invariant if $dr=\ad_{x}^{(2)}r=0$ for any $x\in\mathfrak{g}$. The following is the main result of this section. 

\begin{thm}\label{T1}
Let $\mathfrak{g}$ be a Banach Lie algebra.
\begin{itemize}
\item[{\rm(a)}]  If there is a sequence of symmetric elements $(s_n\in \mathfrak{g}\otimes\mathfrak{g})_{n=0}^{\infty}$ such that $\lim_{n\to +\infty} \ad_{x} s_n=0$ in the projective norm for any $x\in \mathfrak{g}$ and  $(s_n\in \mathfrak{g}\otimes\mathfrak{g})_{n=0}^{\infty}$ is bounded in the projective norm, then $\lim_{n\to +\infty} \ad_{x}\langle s_n,s_n\rangle=0$ in the projective norm.
\item[{\rm(b)}] Let $r=a+s\in \mathfrak{g}\hat{\otimes}_{\pi} \mathfrak{g},$ where $a$ is skew-symmetric and s is symmetric and $\ad$-invariant. Let $(r_n\in \mathfrak{g}\otimes \mathfrak{g})_{n=0}^{\infty}$ be a sequence such that
$\lim_{n\to +\infty} r_n=r$ in the projective norm and $a_n,s_n$ be the skew-symmetric and the symmetric parts of $r_n$ for any $n\geq 0$ respectively. If $\lim_{n\to +\infty}\langle r_n,r_n\rangle=0,$ then $\lim_{n\to+\infty}\ad_{x}[\![a_n,a_n]\!]=0$ in the projective norm for any $x\in \mathfrak{g}$.
\end{itemize}
\end{thm}

\begin{proof}
As $\mathfrak{g}$ is a Banach Lie algebra, there is a $C>0$ such that $\|[x,y]\|\le C\|x\|\ \|y\|$ for any $x,y\in \mathfrak{g}$.

(a)  Let $(s_n\in \mathfrak{g}\otimes\mathfrak{g})_{n=0}^{\infty}$ be a sequence of symmetric elements such that $\lim_{n\to +\infty} \ad_{x}s_n=0$ for any $x\in \frkg$ in the projective norm. Using the Einstein notation, we can write $s_n=\sum_{i} u_{n,i}\otimes v^{n,i}$ as $s_n=u_{n,i}\otimes v^{n,i}$ simply for any $n\ge 0$. By the definition and the symmetry of $s_n$, we have 
\begin{equation}\label{AD3SN}
\begin{aligned}
&d\langle s_n,s_n\rangle(x)=\ad_{x}\langle s_n,s_n\rangle\\
=&\ad_{x}\left([u_{n,i},u_{n,j}]\otimes v^{n,i}\otimes v^{n,j}+u_{n,i}\otimes[v^{n,i}, u_{n,j}]\otimes v^{n,j}+u_{n,i}\otimes u_{n,j}\otimes [v^{n,i}, v^{n,j}]\right)\\
=&\ad_{x}\left([u_{n,i},u_{n,j}]\otimes v^{n,i}\otimes v^{n,j}+v^{n,i}\otimes[u_{n,i}, u_{n,j}]\otimes v^{n,j}+v^{n,i}\otimes v^{n,j}\otimes [u_{n,i}, u_{n,j}]\right).
\end{aligned}
\end{equation}
By the Jacobi identity and the associativity of projective tensor products, we have 
\begin{equation*}
\begin{aligned}
&\ad_{x}\left([u_{n,i},u_{n,j}]\otimes v^{n,i}\otimes v^{n,j}\right)\\
=&[x,[u_{n,i},u_{n,j}]]\otimes v^{n,i}\otimes v^{n,j}+[u_{n,i},u_{n,j}]\otimes [x,v^{n,i}]\otimes v^{n,j}+[u_{n,i},u_{n,j}]\otimes v^{n,i}\otimes [x,v^{n,j}]\\
=&[u_{n,i},[x,u_{n,j}]]\otimes v^{n,i}\otimes v^{n,j}-[u_{n,j},[x,u_{n,i}]]\otimes v^{n,i}\otimes v^{n,j}+[u_{n,i},u_{n,j}]\otimes [x,v^{n,i}]\otimes v^{n,j}\\&+[u_{n,i},u_{n,j}]\otimes v^{n,i}\otimes [x,v^{n,j}].
\end{aligned}
\end{equation*}
It follows that 
\begin{equation*}
\begin{aligned}
&\big{\|}\ad_{x}\left([u_{n,i},u_{n,j}]\otimes v^{n,i}\otimes v^{n,j}\right)\big{\|}_{\pi}\\
\le & \big{\|}[u_{n,i},[x,u_{n,j}]]\otimes v^{n,i}\otimes v^{n,j}+[u_{n,i},u_{n,j}]\otimes v^{n,i}\otimes [x,v^{n,j}] \big{\|}_{\pi}+\big{\|}[u_{n,j},[x,u_{n,i}]]\otimes v^{n,i}\otimes v^{n,j}\\&+[u_{n,j},u_{n,i}]\otimes [x,v^{n,i}]\otimes v^{n,j}\big{\|}_{\pi}\\
\le &\big{\|}\left(\ad_{u_{n,i}}\otimes\id\otimes \id\right)\left(\ad_{x}(u_{n,j}\otimes v^{n,j})\otimes v^{n,i}\right)\big{\|}_{\pi}+\big{\|}\left(\ad_{u_{n,j}}\otimes\id\otimes \id\right)\left(\ad_{x}(u_{n,i}\otimes v^{n,i})\otimes v^{n,j}\right)  \big{\|}_{\pi}\\
\le& 2C^{2}\ \|u_{n,i}\|\;\|v^{n,i}\|\;\|\ad_{x} (u_{n,j}\otimes v^{n,j})\|_{\pi} \le 2C^{2}\ \|s_{n}\|\;\|ds_{n}(x)\|_{\pi}.
\end{aligned}
\end{equation*}
Then by the definition of projective tensor norms, we have 
$$\big{\|}\ad_{x}\left([u_{n,i},u_{n,j}]\otimes v^{n,i}\otimes v^{n,j}\right)\big{\|}_{\pi} \le 2C^{2}\ \|s_{n}\|\;\|ds_{n}(x)\|_{\pi}.$$
Similarly, we have 
$$\left\|\ad_{x} \left(v^{n,i}\otimes[u_{n,i}, u_{n,j}]\otimes v^{n,j}\right)\right\|_{\pi}\le 2C^{2}\ \|s_{n}\|\;\|ds_{n}(x)\|_{\pi}$$ 
and 
$$\left\|\ad_{x}\left(v^{n,i}\otimes v^{n,j}\otimes [u_{n,i}, u_{n,j}]\right)\right\|_{\pi}\le 2C^{2}\ \|s_{n}\|\;\|ds_{n}(x)\|_{\pi}.$$
As $(s_n\in \mathfrak{g}\otimes\mathfrak{g})$ is bounded and $\lim_{n\to +\infty}d s_n=0$ in the projective norm, by \eqref{AD3SN} and the definition of the projective norm, we get $\lim_{n\to +\infty}d\langle s_n,s_n\rangle=0$ in the projective norm.

(b) It is easy to see that $\lim_{n\to \infty}a_n=a$ and $\lim_{n\to\infty}s_n=s$. Using the Einstein notation, for any $n\ge 0,$ denote $a_n=x_{n,i}\otimes y^{n,i}$ and $s_n=u_{n,j}\otimes v^{n,j}$. By the skew-symmetry of $a_n$ and the symmetry of $s_n$, we have 
\begin{equation}\label{3<>}
\begin{aligned}
&\langle r_n,r_n\rangle\\
=&\langle a_n+s_n,a_n+s_n\rangle=\langle a_n,a_n\rangle+\langle s_n,s_n \rangle+[x_{n,i},u_{n,j}]\otimes y^{n,i}\otimes v^{n,j}+[u_{n,j},x_{n,i}]\otimes v^{n,j}\otimes y^{n,i}\\
&+x_{n,i}\otimes [y^{n,i},u_{n,j}]\otimes v^{n,j}+u_{n,j}\otimes [v^{n,j},x_{n,i}]\otimes y^{n,i}+x_{n,i}\otimes u_{n,j}\otimes [y^{n,i},v^{n,j}]+u_{n,j}\otimes x_{n,i}\otimes [v^{n,j},y^{n,i}]\\
=&\langle a_n,a_n\rangle+\langle s_n,s_n \rangle+\left([x_{n,i},u_{n,j}]\otimes y^{n,i}\otimes v^{n,j}+u_{n,j}\otimes y^{n,i}\otimes [x_{n,i},v^{n,j}]\right)-\left([x_{n,i}, u_{n,j}]\otimes v^{n,j}\otimes y^{n,i}\right.\\
&\left.+u_{n,j}\otimes [x_{n,i},v^{n,j}]\otimes y^{n,i}\right)+\left(x_{n,i}\otimes [y^{n,i},u_{n,j}]\otimes v^{n,j}+x_{n,i}\otimes u_{n,j}\otimes [y^{n,i},v^{n,j}]\right)\\
=&\langle a_n,a_n\rangle+\langle s_n,s_n \rangle+\left([x_{n,i},u_{n,j}]\otimes y^{n,i}\otimes v^{n,j}+u_{n,j}\otimes y^{n,i}\otimes [x_{n,i},v^{n,j}]\right)-\left([x_{n,i}, u_{n,j}]\otimes v^{n,j}\otimes y^{n,i}\right.\\
&\left.+u_{n,i}\otimes [x_{n,j},v^{n,i}]\otimes y^{n,j}\right)+\left(x_{n,i}\otimes [y^{n,i},u_{n,j}]\otimes v^{n,j}+x_{n,i}\otimes u_{n,j}\otimes [y^{n,i},v^{n,j}]\right).
\end{aligned}
\end{equation}
Then by the definition and the $\ad$-invariance of $s$, we have
$$
\begin{aligned}
&\big{\|}[x_{n,i},u_{n,j}]\otimes y^{n,i}\otimes v^{n,j}+u_{n,j}\otimes y^{n,i}\otimes [x_{n,i},v^{n,j}]\big{\|}_{\pi}\\
\le& \|y^{n,i}\|\big{\|} \ad_{x_{n,i}} u_{n,j}\otimes v^{n,j}  \big{\|}_{\pi}=\|y^{n,i}\|\big{\|} \ad_{x_{n,i}} s_n  \big{\|}_{\pi} \\
\le&  \|y^{n,i}\|\big{\|} \ad_{x_{n,i}} (s_n-s)  \big{\|}_{\pi}+ \|y^{n,i}\|\big{\|} \ad_{x_{n,i}} s  \big{\|}_{\pi}= \|y^{n,i}\|\big{\|} \ad_{x_{n,i}} (s_n-s)\|_{\pi}\\
\le& 2C\|x_{n,i}\|\ \|y^{n,i}\|\ \|s_n-s\|\le 2C\|s_n\|_{\pi} \ \|s_n-s\|_{\pi}.
\end{aligned}
$$
As $(s_n\in\mathfrak{g}\otimes \mathfrak{g})_{n=0}^{\infty}$ is a Cauchy sequence in the projective norm, it is bounded. Therefore we have 
$$\lim_{n\to +\infty} \big{\|}[x_{n,i},u_{n,j}]\otimes_{\pi} y^{n,i}\otimes_{\pi} v^{n,j}+u_{n,j}\otimes_{\pi} y^{n,i}\otimes_{\pi} [x_{n,i},v^{n,j}]\big{\|}_{\pi}=0.$$ It is similar to verify 
$$\lim_{n\to+\infty}\big{\|}[x_{n,j}, u_{n,i}]\otimes v^{n,i}\otimes y^{n,j}+u_{n,i}\otimes [x_{n,j},v^{n,i}]\otimes y^{n,j}\big{\|}_{\pi}=0$$ and 
$$\lim_{n\to+\infty}\big{\|}x_{n,i}\otimes [y^{n,i},u_{n,j}]\otimes v^{n,j}+x_{n,i}\otimes u_{n,j}\otimes [y^{n,i},v^{n,j}]\big{\|}_{\pi}=0.$$

 Finally as $\lim_{n\to +\infty}\langle r_n,r_n\rangle=0$, by (a), \eqref{STCY} and \eqref{3<>}, we have 
$\lim_{n\to+\infty}\ad_{x}[\![a_n,a_n]\!]=0$ in the projective norm for any $x\in \mathfrak{g}$. 
\end{proof}

By the above theorem, every projective r-matrix $r\in \mathfrak{g}\hat{\otimes}_{\pi}\mathfrak{g}$ with $\ad$-invariant symmetric part gives rise to a projective Lie bialgebra structure on $\mathfrak{g}.$ Furthermore, we have the following result.

\begin{cor}\label{QP}
Let $\mathfrak{g}$ be a Banach Lie algebra and $r=a+s\in \mathfrak{g}\hat{\otimes}_{\pi} \mathfrak{g}$ such that $a$ is the skew-symmetric and  $s$ is the symmetric part of $r$ respectively. If $r$ is a projective {\rm r}-matrix, then $(\mathfrak{g},dr)$ is a projective Banach Lie bialgebra.
\end{cor}

\begin{proof}
By the definition, we know that $dr$ is skew-symmetric. Then by Proposition \ref{JK} and  Theorem \ref{T1}, we find that $dr$ satisfies the Jacobi identity. Finally by Definition \ref{1C}, we obtain that $(\frkg,dr)$ is a projective Banach Lie bialgebra.
\end{proof}

\begin{defi}
Let $\mathfrak{g}$ be a Banach Lie algebra. An element $r\in \mathfrak{g}\hat{\otimes}_{\pi}\mathfrak{g}$ is called a \textbf{quasi-triangular projective r-matrix} if the symmetric part of $r$ is $\ad$-invariant and $r$ is a complete solution of the classical Yang-Baxter equation. Moreover, $r$ is called a \textbf{triangular projective r-matrix} if $r$ is skew-symmetric and $r$ is a projective solution of the classical Yang-Baxter equation.

If $r\in\mathfrak{g}\hat{\otimes}_{\pi} \mathfrak{g}$ is a quasi-triangular r-matrix, then the projective Banach Lie bialgebra $(\mathfrak{g},dr)$ is called \textbf{quasi-triangular}. If $r\in\mathfrak{g}\hat{\otimes}_{\pi} \mathfrak{g}$ is a triangular r-matrix, then the projective Lie bialgebra $(\mathfrak{g},dr)$ is called \textbf{triangular}.   
\end{defi}

Now we give some examples of quasi-triangular projective Lie bialgebras.
\begin{ex}
Let $\mathcal{H}$ be a separable Hilbert space and  $\mathcal{L}(\mathcal{H})$ be the Lie algebra of bounded operators from $\mathcal{H}$ to $\mathcal{H}.$ Then by Example \ref{E1}, 
the operator $\sum_{i=0}^{\infty}  a_{i} (E_{ii}\otimes E_{(i+1)(i+1) }-E_{(i+1)(i+1)}\otimes E_{ii})$ is a projective solution of the classical Yang-Baxter equation. By Corollary \ref{QP}, it induces a triangular projective Banach Lie bialgebra structure on $\mathcal{L}(\mathcal{H}).$
\end{ex}

\begin{ex}
Let $\mathcal{H}$ be a separable Hilbert space.
By Example \ref{E2}, if the sequence $(a_{ij}\in \mathbb{K })_{0\le i<j}^{\infty}$ satisfies
\begin{itemize}
\item[(a)] $\sum_{0\le i<j}^{\infty} |a_{ij}|<\infty$;
\item[(b)] $a_{ij}a_{jk}=a_{ik}a_{ij} + a_{jk}a_{ik},$
\end{itemize}
then the operator $\sum_{0\le i<j}^{\infty}a_{ij} (E_{ij}\otimes E_{ji}-E_{ji}\otimes E_{ji})$ is a projective triangular r-matrix on $\mathcal{L}(\mathcal{H}).$ By Corollary \ref{QP}, the operator induces a projective Banach Lie bialgebra structure on $\mathcal{L}(\mathcal{H}).$  
\end{ex}

\begin{ex}
Let $(\frkg_{i})_{i=1}^{\infty}$
be a family of Banach Lie algebras of finite dimension. Let $(\delta_{i}\in \frkg_{i}\otimes \frkg_{i})$ be a family of skew-symmetric solutions of the projective Yang-Baxter equation and $(a_{i})_{i=1}^{\infty}\in \ell^{1}$. Then by Example \ref{E3}, $\delta=\sum_{i} a_{i}\delta_{i}\in \frkg\otp\frkg$ is a projective triangular r-matrix. It induces a triangular projective Banach Lie bialgebra on $\frkg.$ 
\end{ex}

\subsection{Coboundary Banach Lie bialgebras}
Next, we consider the coboundary analogue of Banach Lie bialgebras, called the coboundary Banach Lie bialgebras.\label{33}
\begin{defi}
	Let $\mathfrak{g}_+$ be a Banach Lie bialgebra with respect to $\mathfrak{g}_-.$ Let $[\cdot,\cdot]_{\mathfrak{g}_-}$ be the Lie bracket of $\mathfrak{g}_-$. Then $\mathfrak{g}_+$ is called a \textbf{coboundary Banach Lie algebra} with respect to $\mathfrak{g}_-$ if $[\cdot,\cdot]_{\mathfrak{g}_-}:\mathfrak{g}_-\times\mathfrak{g}_-\to \mathfrak{g}_-$ restricts to a 1-coboundary $\theta:\mathfrak{g}_+\to \Lambda^{2}\mathfrak{g}_-^{\ast}$ with respect to adjoint action $\ad$ of $\mathfrak{g}_+,$ that is, there is a $S\in\Lambda^{2}\mathfrak{g}_-^{\ast}$ such that 
	$$\theta(x)=\ad_x S,\quad \forall x\in \mathfrak{g}_+.$$
\end{defi}

Next, we establish the relationship between quasi-triangular projective r-matrices and coboundary Banach Lie bialgebras.

\begin{pro}
Let $\frkg$ be a Banach Lie algebra and $r:\frkg\to \frkg\otp\frkg$ be a quasi-triangular {\rm r}-matrix. Then $(\frkg,\frkg^{\ast},p\circ dr)$ is a coboundary Banach Lie bialegbra, where $p$ is the operator given in \eqref{GA}.
\end{pro}

\begin{proof}
It follows from \eqref{DLT} and \eqref{GA} that $p\circ dr$ is a 1-coboundary. This implies $(\frkg,\frkg^{\ast},p\circ dr)$ is a cobuoundary Banach Lie bialgebra.
\end{proof}

\section{Projective Banach Poisson-Lie groups}\label{4}

 In this section, we first introduce the notion of projective Poisson bivectors on Banach manifolds, which leads to the definition of projective Banach Poisson manifolds. Then as the projective corresponding of Banach Poisson-Lie groups, we proposed the notion of projective Banach Poisson Lie groups. As in the case of Banach Poisson Lie groups introduced in \cite{TP}, we prove that the differential of a projective Banach Poisson-Lie group gives rise to a projective Banach Lie bialgebra structure. 

In this section, we shall work in the framework of Banach Poisson manifolds and Banach Lie groups. We do not recall their definitions here. Instead, we refer the readers to \cite{DO} for more details. Unless otherwise stated, we assume all Banach manifolds are smooth.

Let $M$ be a Banach manifold. We will denote by $\barwedge^{2} TM$ the vector bundle over $M$ whose fiber over $a$ lies in the Banach space of skew-symmetric elements of $T_{a} M\otp T_{a} M$. And denote by $\wedge^{2} TM$ the vector bundle over $M$ whose fiber over $a$ lies in the Banach space of skew-symmetric continuous bilinear forms on $M.$

\subsection{Projective Banach Poisson manifolds}

First, considering the Poisson bivectors in the projective tensor setting, we propose the notion of projective Poisson bivectors.

\begin{defi}
Let $M$ be a Banach manifold and $TM,T^{\ast}M$ be the tangent and the cotangent bundles of $M$ respectively. A smooth section $P$ of $\barwedge^{2} TM$ is called a projective Poisson bivector on $M$ if for any closed local sections $\alpha,\beta,\gamma\in T^{\ast}M$,
\begin{equation}\label{PJ}
		\big{(}\alpha \otp d((\beta\otp \gamma\big{)}(P) ))(P)+\big{(}\beta \otp d((\gamma\otp \alpha\big{)}(P) ))(P)+\big{(}\gamma \otp d((\alpha\otp \beta\big{)}(P) ))(P)=0,
	\end{equation}
where $\beta\otp \gamma$ lies in $\wedge^{2} TM$ such that $(\beta\otp \gamma)_{a}(P_{a})=(\beta_{a}\otp\gamma_{a})(P_{a})$ with $\beta_{a}\otp\gamma_{a}$ the unique extension of $\beta\otp\gamma$, and $d((\beta\otp \gamma)(P) ):TM\to \mathbb{K}$ is the differential of $(\beta\otp \gamma)(P).$
\end{defi}

\begin{rmk}
\begin{itemize}
\item[(a)] For a cotangent bundle $T^{\ast}M$, its fibre over a point $a\in M$ is the topological dual space of the tangent space. Different from the case of finite dimensional manifolds, not every element of the cotangent space $T^{\ast}_{a} M$ can be induced by a smooth map $f: M \to \mathbb{R}$ (through its differential $df_{a}: T_{a} M\to \mathbb{R}$). This means the differential of smooth functions only generates a proper subspace of  
$T_{a}^{\ast}M$ in general. Hence the identity \eqref{PJ} implies that the operator $\{\cdot,\cdot\}:\mathcal{C}^{\infty} (M)\times \mathcal{C}^{\infty} (M) \to \mathcal{C}^{\infty} (M)$ given by 
$\{f,g\}=(df\otp dg)(P)$ for any $f,g\in \mathcal{C}^{\infty} (M)$ 
satisfies the Jacobi identity, but the converse is not necessarily true.
\item[(b)]  By \cite[Remark 3.2]{TP}, a covector $\alpha$ in $T^{\ast}_{a} M$ can not be extended to a smooth $1$-form on $M$ in general. While it is always possible to extend $\alpha$ to be a locally defined $1$-form. Thus, as in the case of Banach Poisson manifolds, in the above definition, we used local sections of the cotangent bundle rather than global sections.
\end{itemize}
\end{rmk}

\begin{defi}
A pair $(M, P)$ is called a projective Banach Poisson manifold if $M$ is a Banach manifold and $P$ is a Poisson bivector on $M$.
\end{defi}

Let $(M,P)$ be a projective Banach Poisson manifold. For any $a\in M,$ it follows from \eqref{GA} that $p(P_{a})\in \mathcal{L}(T_{a}M,T_{a}M;\mathbb{K})$ and $p$ is a continuous map from $T_{a}M\otp T_{a}M$ to $\mathcal{L}(T_{a}M,T_{a}M;\mathbb{K})$. Hence $p$ induces a vector bundle morphism from $\barwedge^{2} TM$ to $\wedge^{2} TM$. By abuse of notation, we write the image of $P$ in the induced map as $p\circ P.$

The notions of generalized Banach Poisson manifolds and Banach Poisson-Lie groups were proposed in \cite{TP}. We do not recall them in this note but refer the readers to \cite{TP} for more details. By the definition, it is not difficult to derive the following proposition.

\begin{pro}\label{PGM}
Let $(M, P)$ be a projective Banach Poisson manifold. Then $(M, T^{\ast} M, p\circ P)$ is a generalized Banach Poisson manifold.
\end{pro}

\begin{rmk}
Different from the case of generalized Banach Poisson manifolds, the notion of projective Banach Poisson manifolds coincides with Banach Poisson manifolds introduced in \cite{O}. In fact, let $(M,P)$ be a projective Banach manifold, define $\{\cdot,\cdot\}:\mathcal{C}^{\infty} (M)\times \mathcal{C}^{\infty} (M) \to \mathcal{C}^{\infty} (M)$  by 
$\{f,g\}=(df\otp dg)(P)$ for any $f,g\in \mathcal{C}^{\infty} (M),$ then by \eqref{PJ}, one can check that $X_f=\{f,\cdot\}$ is a smooth vector field on $M$, it is call the \textbf{Hamiltion vector field}.
\end{rmk}

\subsection{Projective Banach Poisson-Lie groups}\label{42}

To define projective Banach Poisson–Lie groups, it is essential to figure out how a projective Poisson structure is constructed on the product of two Poisson manifolds. The following proposition is straightforward.

\begin{pro}\label{PM1}
Let $(M_{1},P')$ and $(M_{2},P'')$ be two projective Banach Poisson manifolds. Then $(M_{1}\times M_{2},P)$ is a projective Banach Poisson manifold, where 
\begin{itemize}
 \item[{\rm(a)}] $M_{1}\times M_{2}$ is the product Banach manifold of $M_{1}$ and $M_{2}.$ Moreover, the tangent bundle of $M_{1}\times M_{2}$ is isomorphic to the direct sum of the vector bundles $TM_{1}$ and $TM_{2}$ and the cotangent bundle of $M_{1}\times M_{2}$ is isomorphic to the direct sum of the vector bundles  $T^{\ast}M_{1}$ and $T^{\ast} M_{2};$
\item[{\rm(b)}] $P$ is given on $M_1\times M_2$ by 
 		\begin{equation*}
 	P_{(a,b)}=P'_{a}+P''_{b},\quad \forall a\in M_{1},\ \forall b\in M_{2}.
 		\end{equation*}
	\end{itemize} 
\end{pro}

\begin{defi}
Let $(M_{1},P_{1})$ and $(M_{2},P_{2})$ be two projective Banach Poisson manifolds and $F:M_{1}\to M_{2}$ be a smooth map. The map $F$ is called a \textbf{projective Poisson map} at $a\in M_{1}$ if 
\begin{equation}\label{PM}
		( T_{a}F\otp  T_{a}F)(P_{a})=P_{F(a)}.
\end{equation}
Moreover, $F$ is called a projective Poisson map if it is a projective Poisson map at any $a\in M_{1}.$
\end{defi}

Next, we introduce the notion of projective Banach Poisson-Lie groups, which coincides with Banach Poisson-Lie groups introduced in \cite{TP}.

\begin{defi}
Let $G$ be a Banach Lie group equipped with a projective Poisson Banach manifold structure $(G,P)$. Then the pair $(G,P)$ is called a \textbf{projective Banach Poisson-Lie group} if the group multiplication map $m: G\times G\to G$ is a projective Poisson map. 
\end{defi}

\begin{ex}
Let $((G_{i},P_{i}))_{i=1}^{\infty}$ be a family of Banach Poisson-Lie groups such that for any $i,$ $G_{i}$ is of finite dimension and $(a_i)_{i=1}^{\infty}\in \ell^{1}$. Let $G=\bigoplus_i G_{i}$. Then $P=\sum_{i} a_{i}P_{i}$ lies in $TG\otp TG$ and $(G,P)$ forms a projective Banach Poisson-Lie group. 
\end{ex}

Denote by $L_{a}$ and $R_{a}$ the left and right translations by $a\in G$ respectively. We extend the notation $L_{a}$ and $R_{a}$ to represent their naturally induced maps on the tangent space $TG$, though technically these are different operations. By abuse of notation, we denote by $L_{a}$ the action of $a\in G$ on a section $P$ of $\barwedge^{2}TM$:
$L_{a}P_{a}=(TL_{a}\otp TL_{a})P,$ where $TL_{a}\otp TL_{a}$ is the unique extension of $TL_{a}\otimes TL_{a}$.
Similarly, we denote by $R_{a}$ the action of $a\in G$ on $P_{a}$:
$R_{a}P_{a}=(TR_{a}\otp TR_{a})P.$

Denote the adjoint action of $G$ on $G$ by $\Ad_{a}=L_{a}\circ R_{a}^{-1}.$ By abuse of notation, we also denote the induced maps of $\Ad_{a}$ on $TG$ and $\barwedge^{2} TG$ by $\Ad_{a}.$

Then in the next proposition, we give an equivalent characterization of the projective Banach Poisson-Lie groups.

\begin{pro}\label{P1}
Let $G$ be a Banach Lie group with an operator $P: G\times G\to G$ such that $(G, P)$ is a projective Banach Poisson manifold. Then $G$ is a projective Banach Lie group if and only if the Poisson tensor  $P$ satisfies
\begin{equation*}
		P_{ab}= L_{a}(P_{b})+ R_{b}(P_{a}),\quad \forall a,b\in G.
\end{equation*} 
\end{pro}

\begin{proof}
By \eqref{PM}, we have
$$P_{ab}=(T_{(a,b)}m\otp T_{(a,b)}m)(P_{(a,b)})$$
	for any $a,b\in G$, where $T_{(a,b)}m\otp T_{(a,b)}m$ is the unique extension of $T_{(a,b)}m\otimes T_{(a,b)}m.$ And it follows from Proposition \ref{PM1} that 
	$$   
	(T_{(a,b)}m\otp T_{(a,b)}m)(P_{(a,b)})=L_{a}(P_{b})+R_{b}(P_{a}).
	$$
Hence $P$ is a projective Poisson bivector if and only if $P_{ab}= L_{a}(P_{b})+R_{b}(P_{a})$. 
\end{proof}

The coherence between the group operation and projective Poisson bivectors can be established by analyzing their induced structures on the Lie algebras. To see this, we give the following equivalent characterization of the projective Banach Poisson-Lie groups.

\begin{pro}\label{DE}
Let $(G, P)$ be a Banach Poisson Lie group equipped with a structure of projective Banach Poisson manifold $(G, P)$. Then $G$ is a projective Banach Poisson-Lie group if and only if the map $\Delta:G\to \barwedge^{2} TG$ given by $\Delta(a)=R_{a^{-1}}P_{a}$ is a 1-cocycle on $G$ with respect to the adjoint action of $G$ on $\barwedge^{2} TG$, that is, for any $a,b\in G,$
	\begin{equation*}
		\Delta(ab)=\Ad (a)\Delta(b)+\Delta(a).
\end{equation*}	
\end{pro}

\begin{proof}
By the definition, we have $R_{(ab)^{-1}}=R_{a^{-1}}\circ R_{b^{-1}}$. Assume that $G$ is a projective Poisson-Lie group, then by Proposition \ref{P1} we see that 
$$\begin{aligned}
R_{(ab)^{-1}}P_{ab}&=R_{a^{-1}}R_{b^{-1}}L_{a}P_{b}+R_{a^{-1}}R_{b^{-1}}R_{b}P_{a}\\
&=R_{a^{-1}}R_{b^{-1}}L_{a}P_{b}+R_{a^{-1}}P_{a}\\
&=R_{a^{-1}}L_{a}R_{b^{-1}}P_{b}+R_{a^{-1}}P_{a}.
\end{aligned}$$ 
Therefore we have $\Delta(ab)=\Ad(a)\Delta(b)+\Delta(a).$
\end{proof}

By the above proposition, one can readily verify the following corollary.   

\begin{cor}\label{DEE}
Let $(G,P)$ be a Banach Poisson Lie group and $\Delta:G\to \barwedge^{2} TG$ be the operator given by 
$\Delta(a)=R_{a^{-1}}P_{a}$ for any $a\in G$. Then $\Delta(e)=0,$ where $e$ is the unit element of $G.$
\end{cor}

In the next proposition, we prove that every projective Banach Poisson-Lie group gives rise to the Banach Poisson Lie group structure.

\begin{pro}
Let $(G,P)$ be a projective Banach Poisson-Lie group. Then $(G, T^{\ast} G, p\circ P)$ is a Banach Poisson Lie group.
\end{pro}

\begin{proof}
The result follows from Proposition \ref{PGM}, Proposition \ref{P1} and \cite[Proposition 5.7]{TP}.
\end{proof}

\begin{lem}\label{DEL}
Let $G$ be a Banach Lie group equipped with the structure of a projective Banach Poisson manifold $(G,P)$. Let $\Delta:G\to \barwedge^{2} TM$ be the operator given by 
$\Delta(a)=R_{a^{-1}}P_{a}$ and $P$ be a smooth local section of $\barwedge ^{2} TM$. Then for any $\alpha,\beta\in \barwedge T^{\ast} M$ around $a\in G$ and $X_{a}\in T_{a} M$, we have  
$$d((\alpha\otp \beta)(P))(X_{a})=(\alpha_{0}(a)\otp \beta_{0}(a))(T_{a}\Delta(X_{a}))+((T_{a}\alpha_{0}(X_{a})\otp \beta_{0}(a))(\Delta(a))+(\alpha_{0}\otp T_{a}\beta_{0}(X_{a}))(\Delta(a))$$
where $T_{a}\Delta$ is the tangent map of $\Delta$ at $a$ and $\alpha_{0}=R_{a}\alpha, \ \beta_{0}=R_{a}\beta.$
\end{lem}

\begin{proof}
The proof follows from the Leibniz rule and the chain law.
\end{proof}

Here is the main result of this section. In the next theorem, we prove that the differential of a projective Banach Poisson-Lie group at the unit element gives rise to a projective Banach Poisson Lie bialgebra structure.

\begin{thm}\label{T2}
 Let $(G,P)$ be a projective Banach Poisson-Lie group and $\Delta:G\to \barwedge^{2} TM$ be the operator given by 
 $$ \Delta(a)=R_{a^{-1}}P_{a},\quad \forall a\in G.$$
 Then $(\mathfrak{g},T_{e} \Delta)$ is a projective Banach Poisson Lie algebra. Here $\mathfrak{g}$ is the Banach Lie algebra of $G$ and $T_{e}\Delta$ is the differential of $\Delta$ at the unit element $e.$
\end{thm}

\begin{proof}
Denote $T_{e} \Delta:\mathfrak{g}\to \mathfrak{g}\otp \mathfrak{g}$ by $\delta$. Let $[\cdot,\cdot]_{\frkg^{\ast}}$ be the operator given by 
\begin{equation*}
	[\alpha_{1},\beta_{1}]_{\frkg^{\ast}}=(\alpha_{1} \otp \beta_{1})(\delta),\quad \forall \alpha_{1},\beta_{1}\in \mathfrak{g}^{\ast},
\end{equation*}
where $\alpha_{1}\otp\beta_{1}$ is the unique extension of $\alpha_{1}\otimes \beta_{1}.$ By the definition, we know that $[\cdot,\cdot]_{\mathfrak{g}^{\ast}}$ is a bilinear map. As $T_{e} \Delta\in \barwedge^{2}TG$ is skew-symmetric, we have $[\cdot,\cdot]_{\mathfrak{g}^{\ast}}$ is also skew-symmetric. Next we show that $[\cdot,\cdot]_{\frkg^{\ast}}$ satisfies the Jacobi identity. By \eqref{PJ}, we get
\begin{equation}\label{JB}
d_{e}((\alpha \otp d((\beta\otp \gamma\big{)}(P) ))(P))+d_{e}((\beta \otp d((\gamma\otp \alpha\big{)}(P) ))(P))+d_{e}((\gamma \otp d((\alpha\otp \beta)(P) ))(P))=0.
\end{equation}

By Corollary \ref{DEE} and Lemma \ref{DEL}, for any $\alpha,\beta\in T^{\ast} G$ we have 
\begin{equation*}
d_{e}((\alpha\otp \beta)(P))=(\alpha(e)\otp \beta(e))(\delta).
\end{equation*}
It follows that 
$$
\begin{aligned}
d_{e}((\alpha \otp d((\beta\otp \gamma)(P) ))(P))&=(\alpha_{1}\otp d_{e}((\beta\otp \gamma)(P) ))(\delta)\\&=(\alpha_{1}\otp (\beta_{1}\otp\gamma_{1}(\delta)))(\delta)\\
&=[[\alpha_{1},\beta_{1}]_{\mathfrak{g}^{\ast}},\gamma_{1}]_{\mathfrak{g}^{\ast}},
\end{aligned}
$$
where $\alpha_{1}=\alpha(e)$ and $\beta_{1}=\beta(e).$
It is similar to obtain that
$$
\begin{aligned}
&d_{e}((\beta \otp d((\gamma\otp \alpha)(P) ))(P))=[[\beta_{1},\gamma_{1}]_{\mathfrak{g}^{\ast}},\alpha_{1}]_{\mathfrak{g}^{\ast}}\\
&d_{e}((\gamma \otp d((\alpha\otp \beta)(P) ))(P))=[[\gamma_{1},\alpha_{1}]_{\mathfrak{g}^{\ast}},\beta_{1}]_{\mathfrak{g}^{\ast}}.\\
\end{aligned}
$$
Then by \eqref{JB}, we know that the Jacobi identity holds. Finally by Proposition \ref{DE}, we have $\delta$ is a 1-cocycle on $\mathfrak{g}^{\ast}.$ Hence $(\mathfrak{g},\delta)$ is a projective Banach Poisson Lie algebra.
\end{proof}

\section{$\mathcal{O}$-operators on Banach Lie algebras}\label{5}

In this section, first we generalize the definition of $\mathcal{O}$-operators to the Banach context. Then we prove that every bounded $\mathcal{O}$-operator on a Banach Lie algebra induces a $1$-coboundary structure. We also use continuous $\mathcal{O}$-operators on Banach Lie algebras to give an equivalent characterization of triangular r-matrices. Finally, we investigate $\mathcal{O}$-operators on the semi-direct product of Banach Lie algebras, and establish their connection with the classical Yang-Baxter equation.

Now we define bounded $\mathcal{O}$-operators on Banach Lie algebras, which is the Banach version of $\mathcal{O}$-operators introduced in \cite{Ku}.

\begin{defi}\label{DO}
Let $\mathfrak{h}$ be a Banach space and $\mathfrak{g}$ be a Banach Lie algebra. Let $\rho:\mathfrak{g}\to \mathcal{B}(\mathfrak{h},\mathfrak{h})$ be a continuous representation. Then $T:\mathfrak{h}\to \mathfrak{g}$ is called a
\textbf{bounded $\mathcal{O}$-operator} on $\mathfrak{g}$ with respect to the representation $(\mathfrak{h},\rho)$ if $T$ is bounded and it satisfies the following condition:
\begin{equation}\label{EDO}
	[T(\alpha),T(\beta)]=T(\rho(T(\alpha))(\beta)-\rho(T(\beta))\alpha  ),\quad \forall \alpha,\beta\in \mathfrak{h}.  
\end{equation}	
 Furthermore, if $\frkh\subset \frkg^{\ast}$, then $T$ is called skew-symmetric if $\alpha(T(\beta))=-\beta(T(\alpha))$ for any $\alpha,\beta\in \mathfrak{h}$.
\end{defi}

We consider \cite[Proposition 2.18]{Ku} in the Banach setting.

\begin{pro}\label {PO}
Let $\mathfrak{g}_+$ be a Banach Lie bialgebra with respect to $\mathfrak{g}_-$ and $T:\mathfrak{g}_+  \to \mathfrak{g}_-$ be a bounded $\mathcal{O}$-operator on $\mathfrak{g}_-$ with respect to $(\mathfrak{g}_- ,-\ad^{\ast}).$ Then $\mathfrak{g}_-$ is a Banach Lie algebra with the Lie bracket given by 	
\begin{equation}\label{OO}
[\alpha,\beta]_{\mathfrak{g}^-}=\ad^{\ast}_{T(\beta)}(\alpha)-\ad^{\ast}_{T(\alpha)}(\beta),\quad \forall \alpha,\beta\in\mathfrak{g}_-.
\end{equation}
\end{pro}

\begin{proof}
By definition, we know that the coadjoint action of $\mathfrak{g}_+$ is continuous, and $\mathfrak{g}_-$ is preserved by the coadjoint action. Then by \cite[Proposition 2.18]{Ku}, we obtain that the bilinear map given above is a Lie bracket on $\mathfrak{g}_-.$ Finally it follows from the continuity of $\ad^{\ast}$ and $T$ that the bilinear map given in the proposition is continuous, therefore $\mathfrak{g}_-$ is a Banach Lie algebra.
\end{proof}

In the next theorem, we investigate the relationship between skew-symmetric bounded operators on Banach Lie algebras and coboundary Banach Lie bialgebras. 

\begin{thm}\label{T3}
Let $\mathfrak{g}_+$ be a Banach Lie algebra and $\langle \cdot,\cdot\rangle_{\mathfrak{g}_+ , \mathfrak{g}_- }$ be a duality pairing between Banach spaces. Let $\mathfrak{g}_-$ be a Banach Poisson Lie space with respect to $\mathfrak{g}_+$ and $T:\mathfrak{g}_-\to\mathfrak{g}_+ $ be a skew-symmetric bounded $\mathcal{O}$-operator on $\mathfrak{g}_-$ with respect to the representation $(\mathfrak{g}_+ ,-\ad^{\ast}).$ Then $\mathfrak{g}_+$ is a coboundary Banach Lie bialgebra with respect to $\mathfrak{g}_-$ such that the operator $S$ given in Definition \ref{DO} lies in $\Lambda^{2} \mathfrak{g}_+  .$
\end{thm}

\begin{proof}
Let $[\cdot,\cdot]_{\mathfrak{g}_-}$ be the Lie bracket given in \eqref{OO}. It follows from Proposition \ref{PO} that $\mathfrak{g}_-$ is a Banach Lie algebra. For any $x\in \mathfrak{g}_+$ and $\alpha,\beta\in \mathfrak{g}_- $ we have 
\begin{equation*}
\begin{aligned}
\langle x, [\alpha,\beta]_{\frkg_{-}}\rangle_{\mathfrak{g}_+ , \mathfrak{g}_-}&=\langle x,\ad^{\ast}_{T(\beta)}(\alpha)-\ad^{\ast}_{T(\alpha)}(\beta)\rangle_{\mathfrak{g}_+ , \mathfrak{g}_- }=\langle [T(\beta),x],\alpha \rangle_{\mathfrak{g}_+ , \mathfrak{g}_- }-\langle [T(\alpha),x],\beta\rangle_{\mathfrak{g}_+ , \mathfrak{g}_- }\\
&=\langle T(\alpha),\ad^{\ast}_{x}\beta\rangle_{\mathfrak{g}_+ , \mathfrak{g}_- }-\langle T(\beta),\ad^{\ast}_{x}\alpha\rangle_{\mathfrak{g}_+ , \mathfrak{g}_- }.
\end{aligned}
\end{equation*}
Then by the skew-symmetry of $T,$ we have 
\begin{equation}\label{EO}
    \langle x,[\alpha,\beta]_{\frkg_{-}} \rangle_{\mathfrak{g}_+ , \mathfrak{g}_{-}}=\langle  T(\ad^{\ast}_{x} \alpha),\beta \rangle_{\mathfrak{g}_+ , \mathfrak{g}_{-}}+\langle T(\alpha),\ad^{\ast}_{x} \beta\rangle_{\mathfrak{g}_{+} , \mathfrak{g}_{-}}.
\end{equation}
Next define $S:\mathfrak{g}_- \times \mathfrak{g}_-\to \mathbb{K}$ by 
$$
S(\alpha,\beta)=\langle T(\alpha),\beta \rangle_{\mathfrak{g}_+ , \mathfrak{g}_- },\quad \forall \alpha,\beta\in\mathfrak{g}_-.$$ 
Then it follows from $\eqref{EO}$ that $\mathfrak{g}_+$ is a coboundary Banach Lie bialgebra. Finally it follows from \eqref{EO} and the skew-symmetry of $T$ that $S\in \Lambda^2 \mathfrak{g}_+.$  
\end{proof}

The following theorem is the main result of this section. Using bounded $\mathcal{O}$-operators, we give an equivalent characterization of skew-symmetric projective r-matrices.

\begin{thm}
Let $\frkg$ be a Banach Lie algebra and $r\in \frkg\otp\frkg.$ Let $S:\mathfrak{g}^{\ast}\times \frkg ^{\ast}\to\mathbb{K}$ be the operator given by 
$$S(\alpha,\beta)=(\alpha\otp \beta)(r),\quad\forall \alpha,\beta\in \frkg^{\ast},$$ 
where $\alpha\otp \beta$ denote the unique extension of $\alpha\otimes \beta.$ If $r$ is skew-symmetric, then $r$ is a quasi-triangular projective {\rm r}-matrix if and only if the operator $T:\mathfrak{g}^{\ast}\to \mathfrak{g}$ given by 
$$T(\alpha)=S_{\alpha},\quad \forall \alpha\in \mathfrak{g}^{\ast} $$ is a bounded $O$-operator on $\mathfrak{g}$ with respect to the coadjoint action $(\mathfrak{g}^{\ast},-\ad^{\ast})$, where $S_{\alpha}=(\alpha\otp \id)(r).$ 
\end{thm}

\begin{proof}
Let $r$ be a triangular projective r-matrix on $\mathfrak{g}.$ There are sequences $(x_n\in \mathfrak{g})$ and $(y_n \in \mathfrak{g})$ such that $\sum_{i=1}^{\infty} \|x_i\| \ \|y
_i\| <\infty$ and $S(\alpha,\beta)=\sum_{i=1}^{\infty}\alpha (x_i)\beta(y_i) .$ Firstly by the definition, we know that $T$ is a bounded operator between Banach spaces $\mathfrak{g}^{\ast}$ and $\mathfrak{g}$. Then using the Einstein notations, we write $r=x_{i}\otimes y^{i}$. We have
$$ [x_i,x_j]\otimes y^{i}\otimes y^{j}+x_i \otimes[y^i , x_j]\otimes y^j + x_i \otimes x_j \otimes [y^i, y^j]=0.$$
As $r$ is skew-symmetric, we get
$$ -[y_i,x_j]\otimes x^{i}\otimes y^{j}+x_i \otimes[y^i , x_j]\otimes y^j + x_i \otimes x_j \otimes [y^i, y^j]=0.$$
It follows that 
$$-\langle x_j \otimes x_i \otimes y^j, \ad^{\ast}_{y^i}\alpha\otimes \beta\otimes\gamma \rangle+\langle x_i \otimes x_j \otimes y^j, \alpha\otimes \ad^{\ast}_{y^{i}}\beta\otimes\gamma \rangle + \langle x_i \otimes x_j \otimes [y^i, y^j],\alpha\otimes \beta\otimes\gamma\rangle=0.$$ Then we find
$$ -(\ad^{\ast}_{\beta(x_i)y_{i}} \alpha)(x_j)\gamma(y^{j}) 
+(\ad^{\ast}_{\alpha(x_i) y^{i} }\beta)(x_j)\gamma(y^j)+\gamma([\alpha(x_i)y^{i},\beta(x_j)y^j]) =0.
$$
By definition, we have
$$-\gamma((\ad^{\ast}_{S_{\beta}} \alpha(x_{j}))y^j)+\gamma((\ad^{\ast}_{S_{\alpha}}\beta(x_j))y^{j})+\gamma([S_{\alpha},S_{\beta}])=0.$$
Then we obtain
$$[T(\alpha),T(\beta)]=T(\ad^{\ast}_{T(\beta)} \alpha-\ad^{\ast}_{T(\alpha)}\beta). $$
Therefore $T$ is an $\mathcal{O}$-operator

Conversely, if the operator $T$ is a bounded $\mathcal{O}$-operator, it is similar to prove that $S$ is a triangular projective r-matrix by the above proof.
\end{proof}

Let $\mathfrak{g}$ be a Banach Lie algebra and $\mathfrak{h}$ be a Banach space. Let $\rho:\mathfrak{g}\to \mathcal{B}(\mathfrak{h},\mathfrak{h})$ be a continuous representation. There is a natural Banach space structure on $\frkg\oplus\frkh$,  where the norm of $\mathfrak{g}\oplus \mathfrak{h}$ is given by 
$$\|x+u\|=\|x\|+\|u\|,\quad\forall u\in \mathfrak{h},\ \forall x\in \mathfrak{g}.$$ 
Then there is a Banach Lie algebra structure on $\mathfrak{g}\oplus \mathfrak{h}$ given by
$$ [x+u,y+v]=[x,y]+\rho(x)v-\rho(y)u,\quad \forall x,y\in \mathfrak{g},\ \forall u,v\in\mathfrak{h}.$$ Denote this Lie algebra by $\mathfrak{g}\ltimes_{\rho} \mathfrak{h}.$

Let $\mathfrak{g}$ and $\mathfrak{h}$ be Banach spaces and $T:\mathfrak{h}\to\mathfrak{g}$ be a linear operator. Then the graph of $T$ is given by
$$\operatorname{Gr}_{T}=\{(T(u),u)\in \mathfrak{g}\oplus \mathfrak{h}:\ \forall u\in\mathfrak{h}\}.$$ 

In the next proposition, we give a graph characterization of bounded $\mathcal{O}$-operators on Banach Lie algebras.

\begin{pro}
Let $\mathfrak{g}$ be a Banach Lie algebra and $\mathfrak{h}$ be a Banach space. Let $\rho:\mathfrak{g}\to \mathcal{B}(\mathfrak{h},\mathfrak{h})$ be a continuous representation. Then a linear operator $T:\mathfrak{h}\to\mathfrak{g}$ is a bounded $\mathcal{O}$-operator on $\mathfrak{g}$ with respect to the representation $(\mathfrak{h},\rho)$ if and only if $\Gr_{\mathfrak{h}}$ is a Banach Lie subalgebra of $\frkg\ltimes_{\rho}\frkh$.
\end{pro}

\begin{proof}
If $T$ is a bounded $O$-operator, for any $T(u)+u, T(v)+v\in \Gr_{T},$ we have
$$
\begin{aligned}
[T(u)+u,T(v)+v]&=[T(u),T(v)]+\rho(T(u))v-\rho(T(v))u\\
&=T(\rho(T(u))v-\rho(T(v))u)+\rho(T(u))v-\rho(T(v))u\in \Gr_{T}.
\end{aligned}
$$ As $T$ is bounded, then by the closed graph theorem, $\Gr_{T}$ is a Banach subspace of $\frkg\oplus \frkh.$

Conversely, if $\Gr_{\mathfrak{h}}$ is a Banach Lie subalgebra of $\frkg\ltimes_{\rho}\frkh$, then for any $u,v\in \mathfrak{h}$, we have 
$$
\begin{aligned}
[T(u)+u,T(v)+v]=[T(u),T(v)]+\rho(T(u))v-\rho (T(v))u\in \Gr_{T} .
\end{aligned}
$$
It follows that $[T(u),T(v)]=T(\rho(T(u)v-\rho(T(v))u)$. Hence $T$ is an $\mathcal{O}$-operator. Finally by the closed graph theorem, we know that $T$ is bounded. 
\end{proof}

Let $\mathcal{H}$ be a separable Hilbert space and $\langle\cdot,\cdot\rangle$ be its inner product. Then there is a complete orthonormal sequence $(e_i)_{i=1}^{\infty}$ such that for any $u\in \mathcal{H},$ 
$$u=\sum_{i=1}^{\infty}  \langle u,e_i\rangle e_i.$$ And by Bessel's inequality, we know that $\sum_{i=1}^{\infty} \langle u,e_i \rangle\le \|u\|^2.$

Let $\mathfrak{g}$ be a Banach Lie algebra and $\rho:\mathfrak{g}\to \mathcal{B}(\mathcal{H},\mathcal{H})$ be a continuous representation of $\mathfrak{g}$ on $\mathcal{H}.$ As $\mathcal{H}$ is self-dual, the dual representation of $\rho$ is $-\rho^{\ast}:\frkg\to \mathcal{B}(\mathcal{H},\mathcal{H})$, where $\rho^{\ast}$ denote the dual map of $\rho.$ 

For any operator $T\in \mathcal{B}(\mathcal{H},\mathfrak{g})$, as $T$ is continuous, we know  that 
$$T(u)=\sum_{i=1}^{\infty}\langle u,e_i \rangle T(e_i),\quad \forall u\in \mathcal{H}.$$ 
Let
$T_n:\mathcal{H}\to \mathfrak{g}$ be the operator given by $$ T_n(u)=\sum_{i=1}^{n} T(e_i),\quad \forall u\in\mathcal{H}.$$ 
Set  
$r^{T_{n}}=\sum_{i=1}^{n} T(e_{i})\otimes e_{i}-e_{i}\otimes T(e_{i}).$ Then we have $r^{T,n}\in \mathfrak{g}\otimes\mathcal{H}\subseteq (\mathfrak{g}\ltimes_{-\rho^{\ast}} \mathcal{H})\otimes (\mathfrak{g}\ltimes_{-\rho^{\ast}} \mathcal{H})$. One can check that $T_n$ is bounded for any $n>0.$ By \cite{Bai}, we have
$\CYB(r^{T_n})$ is a skew-symmetric solution of the classical Yang-Baxter equation if and only if $T_n$ is an $\mathcal{O}$-operator. And we have $\CYB(r^{T_n})\in \mathfrak{g}\otimes \mathcal{H}\otimes\mathcal{H}+\mathcal{H}\otimes\frkg\otimes\mathcal{H}+\mathcal{H}\otimes\mathcal{H}\otimes\frkg$. As $\mathcal{H}$ is self-dual, we see $\CYB(r^{T_n})\in \mathcal{L}(\mathcal{H},\mathcal{H};\mathfrak{g}).$

Finally, we investigate the relationship between $O$-operators with respect to $(\mathcal{H},-\rho^{\ast})$ and the tensor form of the classical Yang-Baxter equation. The next theorem is a generalization of the result in \cite{Bai}.

\begin{thm}
Let $\frkg$ be a Banach Lie algebra and $\mathcal{H}$ be a separable Hilbert space. Let $T:\mathcal{H}\to \frkg$ be a bounded operator, $\rho:\frkg\to\mathcal{H}$ be a continuous representation of $\frkg$ on $\mathcal{H}$ and $\rho^{\ast}$ be the dual representation of $\rho.$ Then $T$ is an $\mathcal{O}$-operator on $\mathfrak{g}$ with respect to the representation $(\mathcal{H},-\rho^{\ast})$ if and only if $\lim_{n\to\infty} \CYB(r^{T_n})(u,v)=0$ for any $u,v\in \mathcal{H}.$
\end{thm}

\begin{proof}
By the Riesz representation theorem, the dual base of $(e_i)_{i=1}^{\infty}$ is $(e_i)_{i=1}^{\infty}.$ Then by the proof of \cite[Claim 1]{Bai}, we have 
\begin{equation*}
\begin{aligned}
\CYB(r^{T_n})=&\sum_{i,j=1}^{n} ([T_{n}(e_i),T_{n}(e_j)]+T_{n}(\rho(T_{n}(e_j))e_{i})-T_{n}(\rho(T_{n}(e_i))e_{j}))\otimes e_{i}\otimes e_{j}\\
&+\sum_{i,j=1}^{\infty} e_{i}\otimes ([T_{n}(e_i),T_{n}(e_j)]+T_{n}(\rho(T_{n}(e_j))e_{i})-T_{n}(\rho(T_{n}(e_i))e_{j}))\otimes e_{j}  \\
&+\sum_{i,j=1}^{\infty} e_{i}\otimes e_{j}\otimes ([T_{n}(e_i),T_{n}(e_j)]+T_{n}(\rho(T_{n}(e_j))e_{i})-T_{n}(\rho(T_{n}(e_i))e_{j})).
\end{aligned}
\end{equation*}
Hence for any $u,v\in \mathcal{H},$ we have
\begin{equation*}
\begin{aligned}
\CYB(r^{T_n})(u,v)=&\sum_{i,j=1}^{n} 3\langle u,e_{i}\rangle \langle v,e_{j} \rangle ([T_{n}(e_i),T_{n}(e_j)]+T_{n}(\rho(T_{n}(e_j))e_{i})-T_{n}(\rho(T_{n}(e_i))e_{j}))\\
=&\sum_{i,j=1}^{n} 3 ([T_{n}(u),T_{n}(v)]+T_{n}(\rho(T_{n}(v)u))-T_{n}(\rho(T_{n}(u)v)))\in \mathfrak{g}.
\end{aligned}
\end{equation*}
As $[\cdot,\cdot]$ and $\rho$ are both continuous maps, we get
$$
\begin{aligned}
&\sum_{i,j=1}^{\infty} ([T_{n}(u),T_{n}(v)]+T_{n}(\rho(T_{n}(v)u)-T_{n}(\rho(T_{n}(u)v))\\
=&[T(u),T(v)]+T(\rho(T(v)u)-T(\rho(T(u)v).
\end{aligned}
$$
Finally we obtain that $T$ is an $\mathcal{O}$-operator if and only of $\lim_{n\to \infty} \CYB(r^{T_n})=0.$
\end{proof}

\vspace{2mm}
\noindent
{\bf Funding Declaration. } This research is supported by Scientific Research Foundation for High-level Talents of Anhui University of Science and Technology (2024yjrc49).

\smallskip

\vspace{2mm}
\noindent
{\bf Declaration of interests. } The authors have no conflicts of interest to disclose.

\smallskip

\vspace{2mm}
\noindent
{\bf Data availability. } Data sharing not applicable to this article as no datasets were generated or analysed in this work.

\vspace{-.2cm}


\end{document}